\documentclass[11pt,english]{smfart}

\usepackage[T2A,T1]{fontenc}
\usepackage[english,francais]{babel}
\usepackage[utf8]{inputenc}

\usepackage{latexsym,amscd,color}
\usepackage{amsmath,amsfonts,amssymb,mathrsfs}
\usepackage{bm}
\usepackage{enumerate,euscript}
\usepackage{amssymb,url,xspace,smfthm,hyperref,CJK}
\usepackage{comment}

\input xy
\xyoption{all}

\numberwithin{equation}{section}

\newcommand{\AMS}{$\mathcal{A}$\mathrm{Ker}n-.1667em\lower.5ex\hbox
        {$\mathcal{M}$}\mathrm{Ker}n-.125em$\mathcal{S}$}

\DeclareMathOperator{\ord}{ord}
\DeclareMathOperator{\mult}{mult}

\DeclareMathOperator{\Spec}{Spec}

\DeclareMathOperator{\vol}{vol}
\def\sbullet{{\scriptscriptstyle\bullet}}

\def\indic{1\hspace{-2.5pt}\mathrm{l}}

\def\esssup_#1{\underset{#1}{\mathrm{ess\,sup\, }}}
\newcommand{\ndot}{\raisebox{.4ex}{.}}

\def\ifquery{\if00}
\ifquery
\def\query#1{\setlength\marginparwidth{65pt} 
\marginpar{\raggedright\fontsize{7.81}{9} 
\selectfont\upshape\hrule\smallskip 
#1\par\smallskip\hrule}} 

\else
\def\query#1{}
\fi

\definecolor{Huayi}{rgb}{0.0, 0.0, 1.0}

\definecolor{Hide}{rgb}{1.0, 0.0, 0.0}

\title{On subfiniteness of graded linear series}
\date{\today}
\author{Huayi Chen}
\address{Univ Paris Diderot, Institut de Math\'ematiques de Jussieu - Paris Rive Gauche, CNRS, Sorbonne Universit\'e  B\^atiment Sophie Campus des Grands Moulins, bâtiment Sophie Germain, case 7012, 75205 Paris cedex 13, France.}
\email{huayi.chen@imj-prg.fr}
\urladdr{webusers.imj-prg.fr/~huayi.chen}
\author{Hideaki Ikoma}
\address{Department of Mathematics, Kyoto University, Kyoto, 606-8502, Japan}
\email{ikoma@math.kyoto-u.ac.jp}
\thanks{Huayi Chen is partially supported by ANR-14-CE25-0015, Hideaki Ikoma is supported by JSPS Grant-in-Aid for Young Scientists (B) No.\ 16K17559 and partially by JSPS Grant-in-Aid (S) No.\ 16H06335.}

\begin{document}
\def\smfbyname{}

\begin{abstract}Hilbert's 14th problem studies the finite generation property of the intersection of an integral algebra of finite type with a subfield of the field of fractions of the algebra. It has a negative answer due to the counterexample of Nagata.
We show that a subfinite  version of Hilbert's 14th problem has a confirmative answer. We then establish a graded analogue of this result, which permits to show that the subfiniteness of graded linear series does not depend on the function field in which we consider it. Finally, we apply the subfiniteness result to the study of geometric and arithmetic graded linear series.
\end{abstract}


\maketitle

\tableofcontents

\section{Introduction}

Let $k$ be a field and $X$ be an integral projective scheme over $\Spec k$. If $D$ is a Cartier divisor on $X$, as a \emph{graded linear series} of $D$, one refers to a graded sub-$k$-algebra of $\bigoplus_{n\in\mathbb N}H^0(X,nD)$. The graded linear series are closely related to the positivity of the divisor and are objects of central interest in the study of the geometry of the underlying polarised scheme $(X,D)$. Classically the asymptotic behaviour of graded linear series of finite type is well understood through the theory of Hilbert polynomials. Several results in birational algebraic geometry, such as Fujita's approximation theorem \cite{Fujita94,Takagi07}, show that certain graded linear series, even though not of finite type, still have a similar asymptotic behaviour as in the finite generation case. More recently, Lazarsfeld-Musta\c{t}\u{a} \cite{Lazarsfeld_Mustata09} and Kaveh-Khovanskii \cite{Kaveh_Khovanskii12,Kaveh_Khovanskii12b} have proposed, after ideas of Okounkov \cite{Okounkov96,Okounkov03}, a method to encode the asymptotic behaviour of dimensions of the homogeneous components of a given graded linear series into a convex body (called the \emph{Newton-Okounkov body}) in an Euclidean space.

Note that a graded linear series of a Cartier divisor is always a graded subalgebra of a graded algebra of finite type. It is then quite  natural to ask if there is a nice birational geometry for algebras of subfinite type (namely subalgebras of an algebra of finite type) over a field. 

From the point of view of birational geometry, it is more convenient to consider graded linear series of a finitely generated field extension $K/k$ without specifying a polarised model of $K$. In this framework, as a \emph{graded linear series} of $K/k$, we refer to a graded sub-$k$-algebra $V_\sbullet$ of the polynomial algebra $K[T]$ such that $V_0=k$ and that $V_n$ is a finite dimensional vector space over $k$ for any $n\in\mathbb N$. In \cite{ChenOk}, a new construction of Newton-Okounkov bodies has been proposed by using ideas from Arakelov geometry, which only depends on a choice of a tower of successive field extensions
$
k=K_0\subset K_1\subset\dots\subset K_d=K
$
such that each extension $K_{i+1}/K_i$ is transcendental and of transcendence degree $1$. The construction is valid for graded linear series of subfinite type (namely contained in a graded linear series of finite type of $K/k$) whose field of rational functions $k(V_\sbullet)$ coincides with $K$ (see Definition \ref{Def:gradedlinearseries}).  One may expect that the same method applies to general graded linear series of subfinite type $V_\sbullet$ by considering $V_\sbullet$ as a graded linear series of $k(V_\sbullet)/k$. However, the main obstruction to this strategy is that \emph{a priori} the condition of subfiniteness depends on  the extension $K/k$ with respect to which we consider the graded linear series. This leads to the following subfiniteness problem: given a graded linear series $V_\sbullet$ of $K/k$ of subfinite type, does there exist a graded linear series $W_\sbullet$ of finite type of the extension $k(V_\sbullet)/k$ which contains $V_\sbullet$?

Note that the above problem is closely related to Hilbert's fourteenth problem\footnote{Let $k$ be a field and $k(x_1,\ldots,x_n)$ be the field of rational functions of $n$ variables.
Hilbert's fourteenth problem asked whether the intersection of a subfield of $k(x_1,\ldots,x_n)$ and the polynomial algebra $k[x_1,\ldots,x_n]$ is finitely generated over $k$ (as a $k$-algebra).}. In fact, given a graded linear series $V_\sbullet$ of $K/k$ which is contained in a graded linear series of finite type $V_\sbullet'$. The intersection of $V_\sbullet'$ with $k(V_\sbullet)[T]$ gives a graded linear series of $k(V_\sbullet)/k$ containing $V_\sbullet$, where $k(V_\sbullet)$ is the field of rational functions of $V_\sbullet$. Unfortunately the intersection is not necessarily a $k$-algebra of finite type, as is shown by Nagata's counterexamples \cite{Nagata60,Nagata59} to Hilbert's fourteenth problem.

Note that the above subfiniteness problem actually asks for a weaker condition than the finite generation of the intersection of $V_\sbullet'$ with $k(V_\sbullet)[T]$. It suffices that the intersection is contained in a graded linear series of finite type of $k(V_\sbullet)$. Similarly, we can consider the following subfinite version of Hilbert's fourteenth problem, which {actually} has a positive answer (see Theorem \ref{Thm:Hilbert14subfinite} and Corollary \ref{Cor:subfinite Hilbert 14} \emph{infra}).

\begin{theo}\label{Thm: subfinite Hilbert 14}
Let $k$ be a field, $R$ be an integral $k$-algebra of finite type and $K$ be the field of fractions of $R$. Let $K'$ be an extension of $k$ which is contained in $K$. Then there exists a finitely generated sub-$k$-algebra $R'$ of $K'$ containing $R\cap K'$, such that $\mathrm{Frac}(R')=\mathrm{Frac}(R\cap K')$. 
\end{theo}

The method of proof consists of an induction argument with respect to the field extension $K/k$ which permits to reduce the problem to the case where the extension $K/k$ is monogenerated. Similar method can be applied to the graded case (but with more subtleties because of the grading structure), which leads to the following result and gives a confirmative answer to the subfiniteness problem of graded linear series. It shows that the subfiniteness of graded linear series is an absolute condition, which does not depend on the choice of field extension with respect to which the graded linear series is considered (see Theorem \ref{Thm: graded case} and Corollary \ref{thm:chens:question} \emph{infra}).

\begin{theo}\label{Thm:subfinite graded}Let $k$ be a field and $K/k$ be a finitely generated field extension. 
Let $V_\sbullet$ be a graded linear series of $K/k$ which is of subfinite type. Then there exists a graded linear series of finite type $W_\sbullet$ of $K/k$ such that $V_\sbullet\subset W_\sbullet$ and $k(V_\sbullet)=k(W_\sbullet)$.
\end{theo}

Recall that Hilbert's fourteenth problem is reformulated in a geometric setting by Zariski \cite{Zariski54}, see also \cite{Nagata57} and the survey article \cite{Mumford_Hilbert}. Note that Theorem \ref{Thm: subfinite Hilbert 14} can be compared with the following result in \cite{Zariski54}.

\begin{theo}[Zariski]
Let $k$ be a field, $A$ an integrally closed $k$-algebra of finite type,  $K:=\mathrm{Frac}(A)$, and  $K'/k$ a subextension of $K/k$.
There then exist an integrally closed $k$-algebra $B$ of finite type and an ideal $I$ of $B$ such that the fraction field of $B$ is $k$-isomorphic to the fraction field of $A\cap K'$ and that
\[
A\cap K'=\bigcup_{n\in\mathbb N}(B:I^n),
\]
where ${(B:I^n):=\{x\in\mathrm{Frac}(B)\,:\,xI^n\subset B \}}$ denotes the ideal quotient.
\end{theo}

Inspired by this result, we establish the following projective version of Zariski's theorem and deduce an alternative proof for Theorem \ref{Thm:subfinite graded} (see Corollary \ref{thm:chens:question} \emph{infra}).

\begin{theo}\label{Thm:Zariski projective}
Let $K/K'/k$ be field extensions of finite type and $W_{\sbullet}$ a graded linear series of $K/k$ that is generated over $k$ by the homogeneous elements of degree $1$.
We assume that $W_1$ contains $1\in K$ and that the projective spectrum $P:=\mathrm{Proj}(W_{\sbullet})$ is a normal scheme.
Let $X$ be any integral normal projective $k$-scheme whose field of rational functions is $k$-isomorphic to $k(W_{\sbullet}\cap K'[T])$.
There then exists a $\mathbb Q$-Weil divisor $D$ on $X$ such that
\[
W_n\cap K'\subset H^0(X,nD)\subset k(W_{\sbullet}\cap K'[T])
\]
for every sufficiently positive $n$.
\end{theo}
 
As an application of the above subfiniteness results, we establish a Fujita approximation theorem for general graded linear series of subfinite type (see Theorem \ref{Theo:Application Fujita approximation} \emph{infra}) and an upper bound for the Hilbert-Samuel function of such graded linear series (see Theorem \ref{Thm: effective upper bound} \emph{infra}). More precisely, we obtain the following results.

\begin{theo}
Let $K/k$ be a finitely generated field extension. For any graded linear series $V_\sbullet$ of $K/k$ of subfinite type, whose Kodaira-Iitaka dimension $d$ is nonnegative, the limit
\[\mathrm{vol}(V_\sbullet)=\lim_{n\in\mathbb N, \,V_n\neq\{0\},\,n\rightarrow+\infty}\frac{\mathrm{dim}_k(V_n)}{n^d/d!}\]
exists in $(0,+\infty)$. Moreover, $\vol(V_\sbullet)$ is equal to the supremum of $\mathrm{vol}(W_\sbullet)$, where $W_\sbullet$ runs over the set of all graded linear series of finite type contained in $V_\sbullet$ having $d$ as the Kodaira-Iitaka dimension. Finally, there exists a function $f:\mathbb N\rightarrow\mathbb R_+$ such that 
\[f(n)=\mathrm{vol}(V_\sbullet)\frac{n^d}{d!}+O(n^{d-1})\]
and that $\mathrm{dim}_k(V_n)\leqslant f(n)$ for any $n\in\mathbb N$.
\end{theo}

We also apply the above results to the study of graded linear series in the arithmetic setting (see Theorem \ref{Thm:convergence thm} \emph{infra}).

The article is organised as follows. In the second section, we prove a weaker form of Hilbert's 14th problem as the subfiniteness result stated in Theorem \ref{Thm: subfinite Hilbert 14}. In the third section, we prove a graded analogue of Theorem \ref{Thm: subfinite Hilbert 14}  in the setting of graded linear series. In the fourth section we consider the subfiniteness problem in the geometric setting as a projective analogue of Zariski's result and establish Theorem \ref{Thm:Zariski projective}. Finally in the fifth section, we develop various applications.


\subsection*{Notation and conventions}
\begin{enumerate}[{\bf 1.}]
\item The field of fractions of an integral domain $A$ is denoted by $\mathrm{Frac}(A)$.
\item Let $K/k$ be an extension of fields. We denote by $\mathrm{tr.deg}_k(K)$ the transcendence degree of $K$ over $k$.
\item Let $S$ be a scheme. For any $i\in\mathbb N$, 
we denote by $S^{(i)}$ the set of points $x$ of $S$ such that the local ring $\mathcal O_{S,x}$ has $i$ as its Krull dimension.
If $S$ is an integral scheme, we denote by $\mathrm{Rat}(S)$ the field of rational functions on $S$.
\item Let $k$ be a field and $X$ be a projective normal scheme over $\Spec k$. As a \emph{Weil divisor} (resp. \emph{$\mathbb Q$-Weil divisor}) on $X$, one refers to an element \[D=\sum_{V\in X^{(1)}}n_VV\] in $\mathbb Z^{\oplus X^{(1)}}$ (resp. $\mathbb Q^{\oplus X^{(1)}}$). 
The coefficient $n_V$ is referred to as the \emph{multiplicity of $D$ along $V$}, and is denoted by $\mult_V(D)$.
If all coefficients $n_V$ are nonnegative, we say that $D$ is \emph{effective}, denoted by $D\geqslant 0$. If $\phi$ is a nonzero rational function on $X$, we denote by $(\phi)$ the principal Weil divisor associated with $\phi$, namely
\[(\phi):=\sum_{V\in X^{(1)}}\mathrm{ord}_V(\phi) V.\]
The map $(\ndot):\mathrm{Rat}(X)^{\times}\rightarrow\mathbb Z^{\oplus X^{(1)}}$ is a group homomorphism and induces a $\mathbb Q$-linear map from $\mathrm{Rat}(X)^{\times}\otimes_{\mathbb Z}\mathbb Q$ to $\mathbb Q^{\oplus X^{(1)}}$ which we denote by $(\ndot)_{\mathbb Q}$. If $D$ is a $\mathbb Q$-Weil divisor on $S$, we define
\begin{equation}
H^0(X,D):=\left\{\phi\in\mathrm{Rat(S)}^{\times}\,:\,D+(\phi\otimes 1)_{\mathbb Q}\geqslant 0\right\}\cup\{0\}
\end{equation}
and
\begin{equation}
R(D)_{\sbullet}:=\bigoplus_{n\geqslant 0}H^0(X,nD)T^n.
\end{equation}
Note that $R(D)_{\sbullet}$ is a graded sub-$k$-algebra of the polynomial algebra $\mathrm{Rat}(X)[T]$.
\item\label{NC:valuations}
Let $K$ be a field. As \emph{discrete valuation} of $K$ we refer to a valuation $\nu:K\rightarrow\mathbb Q\cup\{+\infty\}$ such that $\nu(K^{\times})$ is a monogenerated subgroup of $(\mathbb Q,+)$. Given such a valuation $\nu$, we denote by $O_\nu:=\{f\in K\,:\,\nu(f)\geqslant 0\}$ its valuation ring, $\mathfrak m_\nu$ the maximal ideal of $O_\nu$ and $\kappa(\nu):=O_\nu/\mathfrak m_\nu$ the residue field. If $O_\nu$ {is equal to} $K$, we say that the valuation $\nu$ is \emph{trivial} (note that in this case $\nu(a)=0$ for any $a\in K^{\times}$).

If $K/k$ is a field extension, we call \emph{discrete valuation of $K$ over $k$} any discrete valuation $\nu$ of $K$ such that $\nu(a)=0$ for any $a\in k^{\times}$. In this case $\kappa(\nu)$ is an extension of $k$ and $O_\nu$ is a {$k$}-algebra. Two discrete valuations $\nu_1$ and $\nu_2$ of $K$ over $k$ are said to be \emph{equivalent} if there exists an order-preserving isomorphism $\iota:\nu_1(K^{\times})\to\nu_2(K^{\times})$ such that $\nu_2=\iota\circ\nu_1$. 

Let $K'/k$ be a subextension of $K/k$ and let $\nu$ be a discrete valuation of $K$ over $k$ which is nontrivial.
Then the restriction of $\nu$ to $K'$
is a discrete valuation of $K'$ over $k$.
We define the \emph{ramification index} of $\nu$ with respect to $K'$ as the {unique} integer $e({K',}\nu)\in\mathbb N$ satisfying
\begin{equation}\label{eqn:ramification:index}
\nu(K^{\prime\times})=e(K',\nu)\nu(K^{\times}).
\end{equation}
Note that $e(K',\nu)=0$ if and only if $\nu|_{K'}$ is trivial.

\item\label{NC:centre}
Let $k$ be a field and $S$ be an integral separated $k$-scheme.
Given a discrete valuation $\nu$ of $\mathrm{Rat}(S)$ over $k$, we say that a point $x$ of $S$ is the \emph{centre} of $\nu$ in $S$ if 
\begin{equation}
\mathcal O_{S,x}\subset O_{\nu}\quad\text{and}\quad\mathfrak{m}_{x}=\mathfrak{m}_{\nu}\cap\mathcal O_{S,x},
\end{equation}
where $\mathfrak m_{x}$ denotes the maximal ideal of $\mathcal O_{S,x}$. By the valuative criterion of separation, if the centre of $\nu$ in $S$ exists, then it is unique. In the case where the centre of $\nu$ in $S$ exists, we denote it by $c_S(\nu)$.
If $S$ is proper over $k$, then by the valuative criterion of properness every discrete valuation of $\mathrm{Rat}(S)$ over $k$ has a centre in $S$.

A discrete valuation $\nu$ is trivial if and only if the centre of $\nu$ in $S$ is the generic point. Moreover, each {regular} point $\xi\in S^{(1)}\cup S^{(0)}$ defines a discrete valuation
$\ord_{\xi}:\mathrm{Rat}(S)\to\mathbb Z\cup\{+\infty\}$ whose centre is $\xi$.

\item\label{graded coherent sheaf} Let $R_\sbullet=\bigoplus_{n\in\mathbb N}R_n$ be a graded ring. We denote by $\mathrm{Proj}(R_\sbullet)$ the projective spectrum of $R_\sbullet$. If $M_\sbullet$ is a graded $R_\sbullet$-module, we denote by $\widetilde{M_\sbullet}$ the quasi-coherent $\mathcal O_{\mathrm{Proj}(R_\sbullet)}$-module associated with $M_\sbullet$ (see \cite[\S II.2.5]{EGAII}). For any $m\in\mathbb N$, we let $M(m)_{\sbullet}$ be the $\mathbb N$-graded $R_\sbullet$-module such that $M(m)_n=M_{n+m}$ for any $n\in\mathbb N$, and let $M_{\geqslant m}$ be the $\mathbb N$-graded sub-$R_\sbullet$-module of $M_{\sbullet}$ such that $(M_{\geqslant m})_n=\{0\}$ if $n<m$ and $(M_{\geqslant m})_n=M_n$ if $n\geqslant m$. In particular, one has $M(m)_{\sbullet}=M_{\geqslant m}(m)_{\sbullet}$.
The quasi-coherent sheaf $\widetilde{R(m)_{\sbullet}}$ is denoted by $\mathcal O_{\mathrm{Proj}(R_\sbullet)}(m)$. Note that if $R_\sbullet$ is generated as an $R_0$-algebra by $R_1$, then $\mathcal O_{\mathrm{Proj}(R_\sbullet)}(m)$ are invertible $\mathcal O_{\mathrm{Proj}(R_\sbullet)}$-modules for all $m\in\mathbb N$, and one has canonical isomorphisms
\[\mathcal O_{\mathrm{Proj}(R_\sbullet)}(m)\otimes_{\mathcal O_{\mathrm{Proj}(R_\sbullet)}}\mathcal O_{\mathrm{Proj}(R_\sbullet)}(m')\cong \mathcal O_{\mathrm{Proj}(R_\sbullet)}(m+m')\]
for all $(m,m')\in\mathbb N^2$.
\item\label{essentially integral} {Let $R_\sbullet=\bigoplus_{n\in\mathbb N}R_n$ be a graded ring. We say that $R_\sbullet$ is \emph{essentially integral} if the ideal $R_{\geqslant 1}$ of $R_\sbullet$ does not vanish and if the product of two nonzero homogeneous elements of positive degree of $R_\sbullet$ is nonzero. Note that if $R_\sbullet$ is essentially integral then the scheme $\mathrm{Proj}(R_\sbullet)$ is integral (see \cite[Proposition~II.2.4.4]{EGAII})}
\end{enumerate}

\subsection*{Acknowledgement} Huayi Chen would like to thank Beijing International Center for Mathematical Research for the support of visiting position and for the hospitality. We are grateful to Kiumars Kaveh for communications and comments.
 
\section{A weak form of Hilbert's fourteenth problem}

Let $k$ be a field, $R$ be a finitely generated integral $k$-algebra and $K$ be the field of fractions of $R$. Clearly $K$ is a finitely generated extension of $k$.
Let $K'$ be a subextension of $K/k$, which is also a finitely generated extension (see \cite[Chapitre~V, \S14, $\text{n}^{\circ}$7, Corollaire~3]{Bourbaki_Alg}). We consider the intersection $R\cap K'$ and ask the following question which could be considered as a weaker form of Hilbert's fourteenth problem: \emph{does there exist a finitely generated sub-$k$-algebra $R'$ of $K'$ containing $R\cap K'$ such that $\mathrm{Frac}(R')=\mathrm{Frac}(R\cap K')$}. In this section, we give a confirmative answer to this question.

\begin{defi}
Let $k$ be a field and $A$ be a $k$-algebra. We say that $A$ is of \emph{subfinite type} if it is a sub-$k$-algebra of a $k$-algebra of finite type. 
\end{defi}

\begin{lemm}\label{Lem:injective yields dominant}
An injective homomorphism of rings $A\to B$ yields a dominant morphism $\Spec(B)\to\Spec(A)$.
\end{lemm}

\begin{proof}
Let $\mathfrak{p}$ be a minimal prime ideal of $A$ and $S:=A\setminus\mathfrak{p}$. Since the homomorphism of rings $A\rightarrow B$ is injective, also is the localised homomorphism $A_{\mathfrak{p}}\to S^{-1}B$. Hence $S^{-1}B$ is nonzero. In particular, there exists a prime ideal $\mathfrak{P}$ of $B$ such that $\mathfrak{P}\cap S=\emptyset$, or equivalently, $\mathfrak P\cap A\subset\mathfrak p$.
Since $\mathfrak P\cap A$ is a prime ideal of $A$ and $\mathfrak{p}$ is a  minimal prime ideal of $A$, one has $\mathfrak P\cap A=\mathfrak p$.
\end{proof}

\begin{prop}\label{Pro:contained in domain}
Let $k$ be a field and $A$ be a $k$-algebra of subfinite type. We assume that $A$ is an integral domain. Then there exists a $k$-algebra of finite type containing $A$, which is also an integral domain.
\end{prop}
\begin{proof}
Let $B$ be a $k$-algebra of finite type such that $A\subset B$.
By Lemma \ref{Lem:injective yields dominant}, one can find a prime ideal $\mathfrak p$ of $B$ such that $\mathfrak p\cap A=\{0\}$. Therefore we can consider $A$ as a sub-$k$-algebra of $B/\mathfrak p$.
Since $B$ is a $k$-algebra of finite type, also is $B/\mathfrak p$. The proposition is thus proved.
\end{proof}

\begin{lemm}\label{Lem: finite extension case}
Let $A$ be a $k$-algebra which is an integral domain, and $K$ the field of fractions of $A$. Let $K'/K$ be a finite extension of $K$ generated by one element $\alpha$ and $B'$ a sub-$k$-algebra of finite type of $K'$ which contains $A$. Then there exists a sub-$k$-algebra of finite type $B$ of $K$ which contains $A$.
\end{lemm}
\begin{proof}
 Let $f\in K[T]$ be the minimal polynomial of $\alpha$ over $K$, which is assume to be monic. Let $F_1,\ldots,F_n$ be polynomials in $K[T]$ such that $B'=k[F_1(\alpha),\ldots,F_n(\alpha)]$. Let $S\subset K$ be the (finite) set of the coefficients of the polynomials $f$, $F_1,\ldots,F_n$. We claim that $A$ is contained in $k[S]$. In fact, suppose that an element $u$ of $A$ is written in the form $\varphi(F_1(\alpha),\ldots,F_n(\alpha))$, where $\varphi\in k[X_1,\ldots,X_n]$, then by Euclidean division the polynomial $\varphi(F_1,\ldots,F_n)\in k[S][T]$ can be written as $fg+h$, where $g$ and $h$ are polynomials in $k[S][T]$ with $\deg(h)<\deg(f)$. The decomposition $\varphi(F_1,\ldots,F_n)=fg+h$ is also the Euclidean division of $\varphi(F_1,\ldots,F_n)$ by $f$ in the polynomial ring $K[T]$.  By definition, $\varphi(F_1,\ldots,F_n)-u$ is divisible by $f$ in $K[T]$. Therefore, the polynomial $h$ is actually constant and equals $u$, which shows that $u\in k[S]$.
\end{proof}

\begin{lemm}
Let $A$ be a $k$-algebra which is an integral domain, and $K$ the field of fractions of $A$. Let $K'/K$ be a purely transcendental extension of transcendence degree $1$ and $B'$ a sub-$k$-algebra of finite type of $K'$ which contains $A$. Then there exists a sub-$k$-algebra of finite type $B$ of $K$ which contains $A$.
\end{lemm}
\begin{proof}
Let $\alpha\in K'$ be a transcendental element over $K$ such that $K'=K(\alpha)$. Assume that $B'=k[\varphi_1(\alpha),\ldots,\varphi_n(\alpha)]$, where each $\varphi_i$ is a rational function of the form $F_i/G_i$, where $F_i$ and $G_i$ are polynomials of one variable with coefficients in $K$ and $G_i\neq 0$. Let $\beta$ be an element {in the algebraic closure} of the field $K$ such that $G_i(\beta)\neq 0$ in $K'(\beta)$ for any $i\in\{1,\ldots,n\}$. Then one has $A\subset \widetilde B:=k[\varphi_1(\beta),\ldots,\varphi_n(\beta)]\subset K(\beta)$. In fact, if an element $u$ of $A$ can be written as $P(\varphi_1(\alpha),\ldots,\varphi_n(\alpha))$, where $P$ is a polynomial with coefficients in $k$, then, since $\alpha$ is transcendental over $K(\beta)$, by considering $\alpha$ as {the} variable of rational {functions} and by specifying its value by $\beta$, we obtain that $u=P(\varphi_1(\beta),\ldots,\varphi_n(\beta))$. Finally, by applying Lemma \ref{Lem: finite extension case} to $A\subset \widetilde B$ and the finite extension $K(\beta)/K$, we obtain that there exists a $k$-algebra of finite type $B\subset K$ such that $A\subset B$.
\end{proof}

\begin{theo}\label{Thm:Hilbert14subfinite}
Let $k$ be a field and $A$ be a $k$-algebra of subfinite type. We assume in addition that $A$ is an integral domain and we denote by $K$ the field of fractions of $A$. Then there exists a sub-$k$-algebra of finite type $B$ of $K$ such that $A\subset B$. 
\end{theo}
\begin{proof}
By Proposition \ref{Pro:contained in domain}, there exists a $k$-algebra of finite type $B'$ which is an integral domain containing $A$. Let $K'$ be the field of fractions of $B'$, it is a finitely generated extension of $K$. Therefore there exists a sequence of extensions \[K=K_0\subsetneq K_1\subsetneq\ldots\subsetneq K_n=K'\] such that each extension $K_i/K_{i-1}$ is generated by one element, $i\in\{1,\ldots,n\}$. The extension $K_i/K_{i-1}$ is either generated by an algebraic element over $K_{i-1}$ or is purely transcendental of transcendence degree $1$. By induction we obtain that, for any $i\in\{0,\ldots,n-1\}$, there exists a sub-$k$-algebra of finite type $B_i$ of $K_i$ such that $B_i\supset A$. The theorem is thus proved.
\end{proof}

\begin{coro}\label{Cor:subfinite Hilbert 14}
Let $k$ be a field, $R$ be an integral $k$-algebra of finite type and $K$ be the field of fractions of $R$. Let $K'$ be an extension of $k$ which is contained in $K$. Then there exists a finitely generated sub-$k$-algebra $R'$ of $K'$ containing $R\cap K'$, such that $\mathrm{Frac}(R')=\mathrm{Frac}(R\cap K')$. 
\end{coro}
\begin{proof}
By definition, $R\cap K'$ is an integral $k$-algebra of subfinite type. By Theorem \ref{Thm:Hilbert14subfinite}, there exists a sub-$k$-algebra of finite type $R'$ of $\mathrm{Frac}(R\cap K')$ such that $R\cap K'\subset R'$. Clearly one has $\mathrm{Frac}(R')=\mathrm{Frac}(R\cap K')$ since $R\cap K'\subset R'\subset \mathrm{Frac}(R\cap K')$. The assertion is thus proved. 
\end{proof}

\section{Graded linear series and subfiniteness}

Let $k$ be a field and $K/k$ be a finitely generated field extension. Let \[K[T]=\bigoplus_{n\in\mathbb N}KT^n\]
be the graded ring of polynomials of one variable with coefficients in $K$.

\begin{defi}\label{Def:gradedlinearseries}
As a \emph{graded linear series} of $K/k$ we refer to a graded sub-$k$-algebra \[V_\sbullet=\bigoplus_{n\in\mathbb N}V_nT^n\] of $K[T]$ such that $V_0=k$ and that $V_n$ is a finite dimensional $k$-vector subspace of $K$ for any $n\in\mathbb N_{\geqslant 1}$. 

Let $V_\sbullet$ and $V_\sbullet'$ be two graded linear series of $K/k$. If $V_n\subset V_n'$ for any $n\in\mathbb N$, we say that $V_\sbullet$ \emph{is contained in $V_\sbullet'$}, or $V_\sbullet$ \emph{contains} $V_\sbullet'$, and 
denote it by $V_\sbullet\subset V_\sbullet'$.

Let $V_\sbullet$ be a graded linear series of $K/k$. If $V_\sbullet$ is finitely generated as a $k$-algebra, we say that $V_\sbullet$ is \emph{of finite type}. If $V_\sbullet$ is contained in a graded linear series of finite type, we say that it is \emph{of subfinite type}.

Let $V_\sbullet$ be a graded linear series of $K/k$. We denote by $k(V_\sbullet)$ the subextension of $K/k$ generated by elements of the form $f/g$, where $f$ and $g$ are nonzero elements of $K$ such that there exists $n\in\mathbb N_{\geqslant 1}$ with $\{f,g\}\subset V_n$. The field $k(V_\sbullet)$ is called \emph{the field of rational functions of $V_\sbullet$.}
\end{defi}

\begin{lemm}\label{Rem:generation of homogeneous fractions}
Given any graded linear series $V_{\sbullet}$ of $K/k$, one has
\[
k(V_n)=k(V_{\sbullet})
\]
for every sufficiently positive integer $n$ with $V_n\neq\{0\}$, where $k(V_n)$ denotes the subextension of $K/k$ generated by the elements of the form $f/g$ with $\{f,g\}\subset V_n$, $g\neq 0$.
\end{lemm}

\begin{proof}
First, we note that if $\ell\in\mathbb N_{\geqslant 1}$ is an index such that $V_{\ell}$ contains a nonzero element $h$, then $k(V_m)\subset k(V_{m+\ell n})$ for any $m,n\in\mathbb N_{\geqslant 1}$. In fact, if $\{f,g\}\subset V_m$ and $g\neq 0$, then
\[
\frac{f}{g}=\frac{fh^n}{gh^n}\text{ and }\{fh^n,gh^n\}\subset V_{m+\ell n}
\]
for any $n\in\mathbb N_{\geqslant 1}$.

By changing the grading of $V_{\sbullet}$, we may assume without loss of generality that $\{n\in\mathbb N\,:\,V_n\neq\{0\}\}$ generates $\mathbb Z$ as a $\mathbb Z$-module.
There exist integers $\{n_1,\ldots,n_r\}\subset\mathbb N_{\geqslant 1}$ and nonzero elements $\{f_1,\ldots,f_r, g_1,\ldots,g_r\}\subset K$ such that $\{f_i,g_i\}\subset V_{n_i}$ for any $i\in\{1,\ldots,r\}$ and that  $k(V_\sbullet)=k(f_1/g_1,\ldots,f_r/g_r)$. Set $p:=\mathrm{lcm}(n_1,\ldots,n_r)$.
By the above observation, we can assume $\{f_i,g_i\}\subset V_p$ for any $i$, and one has \[k(V_\sbullet)=k(f_1/g_1,\ldots,f_r/g_r)=k(V_p).\]
Moreover, by the hypothesis that $\{n\in\mathbb N\,:\,V_n\neq\{0\}\}$ generates $\mathbb Z$ as a $\mathbb Z$-module, we can find a positive integer $q$ such that $p$ and $q$ are coprime and that $k(V_p)=k(V_q)=k(V_\sbullet)$.

To conclude the proof, it suffices to show that $\{pm+qn\,:\,m,n\in\mathbb N\}$ contains every sufficiently positive integer.
Since $p$ and $q$ are coprime, we can fix $x,y\in\mathbb Z$ such that $px-qy=1$.
Moreover, we can assume that both $x$ and $y$ are positive.
For any $r$ with $0\leqslant r<q$ and any $n$ with $n\geqslant (q-1)y$,
\[
qn+r=prx+q(n-ry)\in\{pm+qn\,:\,m,n\in\mathbb N\}.
\]
Hence $\{pm+qn\,:\,m,n\in\mathbb N\}$ contains every integer not less than $q(q-1)y$.
\end{proof}

\begin{rema}\label{Rem:twIST}
Let $V_\sbullet$ be a graded linear series of $K/k$ and $f$ be a nonzero element of $K$. We denote by $V_\sbullet(f)$ the graded linear series $\bigoplus_{n\in\mathbb N}f^nV_nT^n$, where $f^nV_n:=\{f^ng\,:\,g\in V_n\}$, called the \emph{twist of $V_\sbullet$ by $f$}. Note that the twist does not change the field of rational functions: one has $k(V_\sbullet(f))=k(V_\sbullet)$ for any $f\in K\setminus\{0\}$.
\end{rema}

\begin{prop}\label{generation higher degree}
Let $W_\sbullet$ be a graded linear series of finite type of $K/k$. Let $n_0$ be an integer such that $n_0\geqslant 1$. There {exist} an integer $r\geqslant 1$ {and} a family $(f_iT^{n_i})_{i=1}^r$ of homogeneous elements in $W_\sbullet$ such that the following conditions are fulfilled:
\begin{enumerate}[(1)]
\item for any $i\in\{1,\ldots,r\}$, one has $n_i\geqslant n_0$;
\item for any integer $n\geqslant n_0$, the vector space {$W_n$} is generated by elements of the form $f_1^{a_1}\cdots f_r^{a_r}$, where $a_1,\ldots,a_r$ are natural numbers such that $a_1n_1+\cdots+a_rn_r=n$.
\end{enumerate}
\end{prop}
\begin{proof}
Suppose that $W_\sbullet$ is generated by $W_1T\oplus\cdots \oplus W_dT^d$. We claim that the graded linear series
\[k\oplus\bigoplus_{n\geqslant n_0}W_nT^n\]
is generated by $W_{n_0}T^{n_0}\oplus\cdots\oplus W_{2n_0+d-2}T^{2n_0+d-2}$. Let $n$ be an integer such that $n\geqslant 2n_0+d-2$. Since $W_\sbullet$ is generated by $W_1T\oplus\cdots\oplus W_dT^d$, we obtain that 
\[W_n=\sum_{\begin{subarray}{c}(a_1,\ldots,a_d)\in\mathbb N^d\\
a_1+2a_2+\cdots+da_d=n
\end{subarray}}W_1^{a_1}\cdots W_d^{a_d}.\]
Let $(a_1,\ldots,a_d)$ be an element in $\mathbb N^d$ such that $a_1+2a_2+\cdots+da_d=n$. Since $n\geqslant 2n_0+d-2$, there exist an integer $m\geqslant 1$ and a family \[\big\{(a_1^{(i)},\ldots,a_d^{(i)})\,:\,i\in\{1,\ldots,m\}\big\}\]
of elements in $\mathbb N^d$ such that  
\[\forall\,j\in\{1,\ldots,d\},\quad a_j^{(1)}+\cdots+a_j^{(m)}=a_j,\]
\[\forall\,i\in\{1,\ldots,m-1\},\quad n_0\leqslant a_1^{(i)}+2a_2^{(i)}+\cdots+da_d^{(i)}\leqslant n_0+d-1,\]
and
\[n_0\leqslant a_1^{(m)}+2a_2^{(m)}+\cdots+da_d^{(m)}\leqslant 2n_0+d-2. \]
Therefore 
\[W_n=\sum_{\begin{subarray}{c}(b_{n_0},\ldots,b_{2n_0+d-2})\in\mathbb N^{n_0+d-1}\\
n_0b_{n_0}+\cdots+(2n_0+d-2)b_{2n_0+d-2}=n
\end{subarray}}W_{n_0}^{b_{n_0}}\cdots W_{2n_0+d-2}^{b_{2n_0+d-2}},\]
which concludes the claim {($b_j$ corresponds to the number of $i\in\{1,\ldots,m\}$ such that $a_1^{(i)}+2a_2^{(i)}+\cdots+da_d^{(i)}=j$)}. Finally it suffices to choose a family of homogeneous elements in $W_\sbullet$ which forms a basis of $W_{n_0}T^{n_0}\oplus\cdots\oplus W_{2n_0+d-2}T^{2n_0+d-2}$.
\end{proof}

\begin{lemm}\label{Lem: de type fini corps}
Let $K/k'/k$ be extensions of fields. We assume that the extension $K/k$ is finitely generated and the extension 
$k'/k$ is finite. Let $W_\sbullet'$ be a graded linear series of finite type of $K/ k'$ and let 
\[W_\sbullet=k\oplus\bigoplus_{n\in\mathbb N_{\geqslant 1}}W_n'T^n.\]
Then $W_\sbullet$ is a graded linear series of finite type of $K/k$.
\end{lemm}
\begin{proof}
Let $(f_iT^{n_i})_{i=1}^r$ be a system of generators of $W_{\sbullet}'$. Let $(\theta_j)_{j=1}^m$ be a basis of $k'$ over $k$. We claim that $W_\sbullet$ is generated by 
\begin{equation}\label{equ:generator theta}(\theta_jf_iT^{n_i})_{(i,j)\in\{1,\ldots,r\}\times\{1,\ldots,m\}}.\end{equation}
In fact, if $\varphi$ is an element of $W_n'$, then it can be written as
\[\sum_{\begin{subarray}{c}
\boldsymbol{a}=(a_1,\ldots,a_r)\in\mathbb N^r\\ a_1n_1+\cdots+a_rn_r=n
\end{subarray}}\lambda_{\boldsymbol{a}}f_1^{a_1}\cdots f_r^{a_r},\]
where the coefficients $\lambda_{\boldsymbol{a}}$ belong to $k'$. By writing $\lambda_{\boldsymbol{a}}$ as a linear combination of $(\theta_j)_{j=1}^m$, we obtain that $\varphi$ lies in the graded linear series of $K/k$ generated by \eqref{equ:generator theta}. The lemma is thus proved.
\end{proof}

\begin{defi}\label{Def: Kodaira-Iitaka}
Let $V_\sbullet$ be a graded linear series of $K/k$. We assume that there exists $n\in\mathbb N_{\geqslant 1}$ such that $V_n\neq\{0\}$. We define the \emph{Kodaira-Iitaka dimension} of $V_\sbullet$ as the transcendence degree of $k(V_\sbullet)$ over $k$. We refer the readers to \cite[\S3]{Kaveh_Khovanskii12b} and \cite[\S2]{Cutkosky14} for the definition of Kodaira-Iitaka dimension in the setting of graded linear series of Cartier divisors or line bundles. If $V_n=\{0\}$ for any $n\in\mathbb N_{\geqslant 1}$, then by convention the \emph{Kodaira-Iitaka dimension} of $V_\sbullet$ is defined to be $-\infty$.
\end{defi}

\begin{theo}\label{Thm: graded case}
Let $V_\sbullet$ be a graded linear series of $K/k$. Assume that there exists a graded linear series of finite type $V_\sbullet'$ of $K/k$ which contains $V_\sbullet$. Then there exists a graded linear series of finite type $W_\sbullet$ of $K/k$ such that $V_\sbullet\subset W_\sbullet$ and $k(V_\sbullet)=k(W_\sbullet)$.
\end{theo}
\begin{proof}
\emph{Step 1: reduction to the case where $1\in V_1$ and $k(V_1')=k(V_\sbullet')$.} Let $\Theta:=\{n\in\mathbb N_{\geqslant 1}\,:\,V_n\neq\{0\}\}$. The assertion of the theorem is trivial when $\Theta=\varnothing$. In the following, we assume that $\Theta$ is not empty, and hence it is a subsemigroup of $\mathbb N_{\geqslant 1}$. Let $a\in\mathbb N_{\geqslant 1}$ be a generator of the subgroup of $\mathbb Z$ generated by $\Theta$. As $\bigoplus_{n\in\mathbb N}V_{an}'T^{an}$ is a $k$-algebra of finite type (see for example \cite[Lemme~II.2.1.6.(iv)]{EGAII}), by changing the grading we can reduce the problem to the case where $a=1$. In particular, there exists an $m\in\mathbb N_{\geqslant 1}$ such that the vector spaces $V_m$ and $V_{m+1}$ are both nonzero. We pick $x\in V_m\setminus\{0\}$ and $y\in V_{m+1}\setminus\{0\}$. By replacing $V_\sbullet$ by the graded linear series generated by $V_\sbullet$ and $(y/x)T$ and replacing $V_\sbullet'$ by the graded linear series generated by $V_\sbullet'$ and $(y/x)T$ (this procedure does not change the fields of {rational functions}), we reduce the problem to the case where $V_1\neq\{0\}$. Finally, by replacing $V_\sbullet$ by $V_\sbullet(f^{-1})$ and $V_\sbullet'$ by $V_\sbullet'(f^{-1})$ {(see Remark \ref{Rem:twIST} for the notation)}, where $f$ is a nonzero element of $V_1$ (again this procedure does not change the fields of fractions, see Remark \ref{Rem:twIST}), we reduce the problem to the case where $1\in V_1$. Moreover, by replacing $V_\sbullet'$ by the graded linear series generated by $V_\sbullet'$ and $\alpha_1T,\ldots,\alpha_mT$, where $\{\alpha_1,\ldots,\alpha_m\}$ is a system of generators of $k(V_\sbullet')$ over $k$, we may assume that $k(V_1')=k(V_\sbullet')$.

\emph{Step 2: reduction to the simple extension case by induction.} As explained in the previous step, we can assume $1\in V_1$ and $k(V_1')=k(V_\sbullet')$. Since $k(V_\sbullet')/k(V_\sbullet)$ is a finitely generated extension of fields (where $V_1$ is assumed to contain $1$), there exist successive extensions of fields
\[k(V_\sbullet)=K_0\subsetneq K_1\subsetneq\ldots\subsetneq K_b=k(V_\sbullet')\]
such that each extension $K_i/K_{i-1}$ is generated by one element of $V_1'$. 

Assume that the theorem has been proved for the case where $k(V_\sbullet')/k(V_\sbullet)$ is generated by one element in $V_1'$. Then by induction we can show that, for any $i\in\{0,\ldots,b\}$, there exists a graded linear series of finite type $W_\sbullet^{(i)}$, which contains $V_\sbullet$ and such that $k(W_\sbullet^{(i)})=K_i$. In fact, we can choose $W^{(r)}_\sbullet=V_\sbullet'$. Assume that we have chosen a graded linear series of finite type $W_\sbullet^{(i+1)}$  such that $W_\sbullet^{(i+1)}\supset V_\sbullet$ and $k(W_\sbullet^{(i+1)})=K_{i+1}$, where $i\in\{0,\ldots,b-1\}$. Let $V^{(i)}_\sbullet$ be the graded linear series generated by $V_\sbullet$ and a finite system of generators of $K_i/k$ in $V_1'$. The graded linear series $V^{(i)}_\sbullet$ contains $V_\sbullet$ and $K_i=k(V^{(i)}_1)$. Without loss of generality we may assume that $V^{(i)}_\sbullet\subset W^{(i+1)}_\sbullet$ and that the extension $K_{i+1}/K_i$ is generated by one element $\alpha$ in $W_1^{(i+1)}$, otherwise we just replace $W^{(i+1)}_\sbullet$ by the graded linear series generated by $W^{(i+1)}_\sbullet$, $V^{(i)}_1$ and a generator of the extension $K_{i+1}/K_i$ in $V_1'$. It is a graded linear series of finite type which contains $V_\sbullet$ and has $K_{i+1}$ as its field of rational functions. If the theorem has been proved for the simple extension case, then we obtain the existence of a graded linear series of finite type $W_\sbullet^{(i)}$ such that $V_\sbullet\subset W_\sbullet$ and $k(W_\sbullet^{(i)})=K_i$.

Note that the graded linear series $W_\sbullet=W_\sbullet^{(0)}$ satisfies the conditions $V_\sbullet\subset W_\sbullet$ and $k(V_\sbullet)=k(W_\sbullet)$. Therefore, to prove the theorem it suffices to prove the particular case where the extension $k(V_\sbullet')/k(V_\sbullet)$ is generated by one element in $V_1'$. Similarly, to prove the theorem under the supplementary condition that the extension $k(V_\sbullet')/k(V_\sbullet)$ is algebraic, it suffices to prove the particular case where the extension $k(V_\sbullet')/k(V_\sbullet)$ is generated by one element in $V_1'$ which is algebraic over $k(V_\sbullet)$. 

\emph{Step 3: algebraic extension case.}
In this step, we prove the theorem under the assumption that the extension $k(V_\sbullet')/k(V_\sbullet)$ is algebraic. As explained in the previous two steps, we may suppose without loss of generality that $1\in V_1$, $k(V_1')=k(V_\sbullet')$ and the extension $k(V_\sbullet')/k(V_\sbullet)$ is generated by one element $\alpha$ in $V_1'$ which is algebraic over $k(V_\sbullet)$.

 Let 
\[G(X):=X^{\delta}+\xi_1X^{\delta-1}+\cdots+\xi_\delta\in k(V_\sbullet)[X]\]
be the minimal polynomial of $\alpha$ over $k(V_\sbullet)$. By Proposition \ref{generation higher degree}, there exist an integer $r\in\mathbb N_{\geqslant 1}$ and  homogeneous elements $(f_iT^{n_i})_{i=1}^r$ with $n_i\geqslant \delta$ for any $i\in\{1,\ldots,r\}$, which generates the graded linear series
\[k\oplus\bigoplus_{n\geqslant \delta}V_n'T^n.\]
Since $1\in V_n\subset V_n'$ for any $n\in\mathbb N_{\geqslant 1}$, for any $i\in\{1,\ldots,r\}$, one has $f_i\in k(V_\sbullet')$. Moreover, since the extension $k(V_\sbullet')/k(V_\sbullet)$ is generated by $\alpha$ (which is of degree $\delta$ over $k(V_\sbullet)$), there {exist} polynomials \[F_i(X):=\eta_{i,1}X^{\delta-1}+\cdots+\eta_{i,\delta}\in k(V_\sbullet)[X],\quad i\in\{1,\ldots,r\}\]
such that $f_i=F_i(\alpha)$ for any $i\in\{1,\ldots,r\}$. We introduce the following polynomials in $k(V_\sbullet)[T,Y]$
\begin{gather*}
\widetilde G(T,Y)=Y^\delta+(\xi_1T)Y^{\delta-1}+\cdots+\xi_\delta T^\delta,\\
\widetilde F_{i}(T,Y)=(\eta_{i,1}T^{n_i-\delta+1})Y^{\delta-1}+\cdots+\eta_{i,\delta}T^{n_i}.
\end{gather*} 
Note that one has $\widetilde G(T,TX)=G(X)T^\delta$ and $\widetilde F(T,TX)=F_i(X)T^{n_i}$.

We let $W_\sbullet$ be the graded linear series generated by $V_1T\oplus\cdots\oplus V_{\delta-1}T^{\delta-1}$ and the elements $\xi_1T,\ldots,\xi_\delta T^{\delta}$, $\eta_{i,1}T^{n_i-\delta+1},\ldots,\eta_{i,\delta}T^{n_i}$ ($i\in\{1,\ldots,r\}$). It is a graded linear series of finite type of $K/k$ such that $k(W_\sbullet)\subset k(V_\sbullet)$. It remains to prove that $W_\sbullet$ contains $V_\sbullet$. Clearly $V_n\subset W_n$ for $n\in\{1,\ldots,\delta-1\}$. Let $n\in\mathbb N_{\geqslant\delta}$ and $\varphi$ be an element in $V_n\subset V_n'$. By definition $\varphi$ can be written in the form
\[\sum_{\begin{subarray}{c}\boldsymbol{a}=(a_1,\ldots,a_r)\in\mathbb N^r\\
a_1n_1+\cdots+a_rn_r=n
\end{subarray}}\lambda_{\boldsymbol{a}}f_1^{a_1}\cdots f_r^{a_r}=\sum_{\begin{subarray}{c}\boldsymbol{a}=(a_1,\ldots,a_r)\in\mathbb N^r\\
a_1n_1+\cdots+a_rn_r=n
\end{subarray}}\lambda_{\boldsymbol{a}}F_1(\alpha)^{a_1}\cdots F_r(\alpha)^{a_r},\]
where $\lambda_{\boldsymbol{a}}\in k$. We consider the element
\[\widetilde F(T,Y)=\sum_{\begin{subarray}{c}\boldsymbol{a}=(a_1,\ldots,a_r)\in\mathbb N^r\\
a_1n_1+\cdots+a_rn_r=n
\end{subarray}}\lambda_{\boldsymbol{a}}\widetilde F_1(T,Y)^{a_1}\cdots \widetilde F_r(T,Y)^{a_r}\in k(V_\sbullet)[T,Y].
\]
Viewed as a polynomial on $Y$ with coefficients in $k(V_\sbullet)[T]$, the coefficients of $\widetilde F(T,Y)$ can be written as the values of certain polynomials  on $\eta_{i,1}T^{n_i-\delta+1},\ldots,\eta_{i,\delta}T^{n_i}$ ($i\in\{1,\ldots,r\}$).
Note that one has \[\widetilde F(T,TX)=\sum_{\begin{subarray}{c}\boldsymbol{a}=(a_1,\ldots,a_r)\in\mathbb N^r\\
a_1n_1+\cdots+a_rn_r=n
\end{subarray}}\lambda_{\boldsymbol{a}}F_1(X)^{a_1}\cdots F_r(X)^{a_r}T^n.\]
Therefore $\widetilde F(T,T\alpha)-\varphi T^n=0$ in $k(V_\sbullet')[T]$. Since $G$ is the minimal polynomial of $\alpha$, an Euclidean division argument shows that $\varphi T^n$ can be written as a polynomial of $\xi_1T,\ldots,\xi_\delta T^{\delta}$, $\eta_{i,1}T^{n_i-\delta+1},\ldots,\eta_{i,\delta}T^{n_i}$  ($i\in\{1,\ldots,r\}$) with coefficients in $k$. The theorem is thus proved in the particular case where $k(V_\sbullet')/k(V_\sbullet)$ is an algebraic extension.

\emph{Step 4: general case.} In this step, we prove the theorem in the general case. As explained in steps 1 and 2, we may assume that $1\in V_1$, $k(V_1')=k(V_\sbullet')$ and that the extension $k(V_\sbullet')/k(V_\sbullet)$ is generated by one element $\alpha$ in $V_1'$ which is transcendental over $k(V_\sbullet)$ (the algebraic case has already been treated in Step 3). 

Since $V_\sbullet'$ is of finite type, there exist an integer $r\geqslant 1$ and homogeneous elements $(f_iT^{n_i})_{i=1}^r$ which generate $V_\sbullet'$ as a $k$-algebra. As $k(V_\sbullet')/k(V_\sbullet)$ is generated by $\alpha$, there exists rational functions $P_i/Q_i$ ($i\in\{1,\ldots,r\}$), where $\{P_i,Q_i\}\subset k(V_\sbullet)[X]$, $Q_i\neq 0$, such that $f_i=P_i(\alpha)/Q_i(\alpha)$. 

Let $\theta$ be an element in the algebraic closure $\overline k$ of $k$, such that $Q_i(\theta)\neq 0$ for any $i\in\{1,\ldots,r\}$. Let $\widehat k=k(\theta)$ and $\widehat K=K(\theta)$. Then $\widehat K/K$ is a finite extension of field, and $\widehat K/\widehat k$ is a purely transcendental extension generated by $\alpha$. Let $\widehat V_\sbullet$ and $\widehat V_\sbullet'$ be the graded sub-$\widehat k$-algebra of $\widehat K[T]$ generated by $V_\sbullet$ and $V_\sbullet'$ respectively. Then $\widehat V_{\sbullet}'$ is generated as a $\widehat k$-algebra by $(f_iT^{n_i})_{i=1}^r$. We let $\widehat W_\sbullet$ be the graded linear series of $\widehat K/\widehat k$ generated by $T$ and elements of the form $(P_i(\theta)/Q_i(\theta))T^{n_i}$, where $i\in\{1,\ldots,r\}$. This is a graded linear series of finite type. Note that $P_i(\theta)/Q_i(\theta)\in \widehat k(\widehat V_\sbullet)$ for any $i\in\{1,\ldots,r\}$. Therefore $\widehat k(\widehat W_\sbullet)\subset \widehat k(\widehat V_\sbullet)$.

Let $n\in\mathbb N_{\geqslant 1}$ and $\varphi$ be an element of $\widehat V_n\subset \widehat V_n'$. By definition $\varphi$ can be written in the form
\[\varphi=\sum_{\begin{subarray}{c}\boldsymbol{a}=(a_1,\ldots,a_r)\in\mathbb N^r\\
a_1n_1+\cdots+a_rn_r=n
\end{subarray}}\lambda_{\boldsymbol{a}}f_1^{a_1}\cdots f_r^{a_r},\]
where the coefficients $\lambda_{\boldsymbol{a}}$ belong to $\widehat k$.
As $\alpha$ is transcendental over $\widehat k(\widehat V_\sbullet)$, we obtain that
\[\varphi=\sum_{\begin{subarray}{c}\boldsymbol{a}=(a_1,\ldots,a_r)\in\mathbb N^r\\
a_1n_1+\cdots+a_rn_r=n
\end{subarray}}\lambda_{\boldsymbol{a}}\prod_{i=1}^r\Big(\frac{P_i(\theta)}{Q_i(\theta)}\Big)^{a_i},\]
which shows that ${\varphi}\in W_n$. Therefore one has $\widehat V_\sbullet\subset \widehat W_\sbullet$, which implies that $\widehat k(\widehat V_\sbullet)=\widehat k(\widehat W_\sbullet)$ since we have already seen that $\widehat k(\widehat W_\sbullet)\subset \widehat k(\widehat V_\sbullet)$.

Let 
\[W_\sbullet':=k\oplus\bigoplus_{n\in\mathbb N_{\geqslant 1}}\widehat W_nT^n.\]
Since $\widehat W_\sbullet$ is a graded linear series of finite type of $\widehat K/\widehat k$, by Lemma \ref{Lem: de type fini corps} we obtain that $W_\sbullet'$ is a graded linear series of $\widehat K/k$ of finite type. Moreover, one has $V_\sbullet\subset W_\sbullet'$ and $k(W_\sbullet')\subset \widehat k(\widehat W_\sbullet)=\widehat k(\widehat{V}_\sbullet)$ is a finite extension of $k(V_\sbullet)$. Therefore, by the algebraic extension case of the theorem proved in Step 3 we obtain the existence of a graded linear series of finite type $W_\sbullet$ of $\widehat K/k$ such that $V_\sbullet\subset W_\sbullet$ and that $k(V_\sbullet)=k(W_\sbullet)$. Moreover, the {equality $k(V_\sbullet)=k(W_\sbullet)$} and the assumption $1\in V_1\subset W_1$ imply that $W_\sbullet$ is a graded linear series of $k(V_\sbullet)/k$ (and hence a graded linear series of $K/k$). The theorem is thus proved.
\end{proof}

\section{A subfinite version of Zariski's theorem}

\subsection{Preliminaries}

In this section, we collect several basic facts on the valuations and on the graded rings, which we use to show Theorem \ref{Thm:Zariski projective}.

\subsubsection{}
Basic notions of valuations and their centres are in Notation and conventions \ref{NC:valuations} and \ref{NC:centre}.

\begin{lemm}\label{lem: centre of the restriction}
Let $\pi:X\to X'$ be a dominant morphism of integral separated $k$-schemes, $K:=\mathrm{Rat}(X)$, $K':=\mathrm{Rat}(X')$, and $\nu$ a discrete valuation of $K$ over $k$.
If the centre $c_X(\nu)$ of $\nu$ in $X$ exists, then $\pi(c_X(\nu))$ is the center of $\nu|_{K'}$ in $X'$, namely $\pi(c_X(\nu))=c_{X'}(\nu|_{K'})$.
\end{lemm}

\begin{proof}
Since the morphism $\pi$ is dominant, it induces an injective homomorphism of fields $\mathrm{Rat}(X')\rightarrow\mathrm{Rat}(X)$, which allows to consider $K'$ as a subfield of $K$. Recall that the centre $c_X(\nu)$ is the unique point $x\in X$ satisfying $\mathcal O_{X,x}\subset O_{\nu}$ and $\mathfrak m_x=\mathfrak m_{\nu}\cap\mathcal O_{X,x}$ (see Notation and conventions {\bf\ref{NC:centre}}).
Note that 
\[
O_{\nu|_{K'}}=\left\{f\in K'\,:\,\nu(f)\geqslant 0\right\}=O_{\nu}\cap K',\quad\text{and}\quad \mathfrak m_{\nu|_{K'}}=\mathfrak m_{\nu}\cap K'.
\]
Hence $O_{X',\pi(c_X(\nu))}\subset O_{\nu|_{K'}}$ and $\mathfrak m_{\pi(c_X(\nu))}\subset\mathfrak m_{\nu|_{K'}}$ (which implies $\mathfrak m_{\pi(c_X(\nu))}=\mathfrak m_{v|_{K'}}\cap\mathcal O_{X',\pi(c_X(\nu))}$ since $\mathfrak m_{\pi(c_X(\nu))}$ is a maximal ideal).
\end{proof}

\begin{lemm}\label{lem: extension of valuation}
Let $K/K'$ be a field extension of finite type.
Then any discrete valuation $\nu'$ of $K'$ extends to at least one discrete valuation $\nu$ of $K$  such that the following diagram is commutative
\[
\xymatrix{K^{\prime\times} \ar[r]^-{\nu'} \ar[d] & \mathbb{Q} \\ 
K^{\times} \ar[ur]_-{\nu}
}
\]
(see Notation and conventions \ref{NC:valuations}).
\end{lemm}

\begin{proof}
By induction it suffices to treat the case where the extension $K/K'$ is generated by one element $\alpha$. If $\alpha$ is transcendental over $K'$, then $K=K'(\alpha)$ is canonically isomorphic to the field of rational function in one variable. Therefore the valuation $\nu:K\rightarrow\mathbb Q\cup\{+\infty\}$ such that
\[\nu(a_0+a_1\alpha+\cdots+a_n\alpha^n)=\min(\nu'(a_0),\ldots,\nu'(a_n))\]
for any $a_0+a_1X+\cdots+a_nX^n\in K'[X]$ is a valuation extending $\nu'$. The valuations $\nu'$ and $\nu$ have the same image and hence $\nu$ is discrete.

Assume that $\alpha$ is algebraic over $K'$. Let $\widehat{K'}$ be the  completion of $K'$ with respect to $\nu'$, on which the valuation $\nu'$ extends in a unique way. We choose an embedding of $K$ in the algebraic closure $\widehat{K'}{}^{\mathrm{a}}$ of $\widehat{K'}$ and let $L$ be the subfield of $\widehat{K'}{}^{\mathrm{a}}$ generated by $\widehat{K'}$ and $K$. Then $L$ is a finite extension of $\widehat{K'}$, on which there is a unique valuation $\omega$ extending $\nu'$ such that 
\[\forall\,x\in L,\quad \omega(x):=\frac{1}{[K:\widehat{K'}]}\nu'(\mathrm{Norm}_{L/\widehat{K'}}(x)).\]
Let $\nu$ be the restriction of $\omega$ on $K$. It is a valuation extending $\nu'$. Moreover, it is discrete since $\nu(K^{\times})\subset \frac{1}{[K:\widehat{K'}]}\nu'(K'{}^{\times})$.
\end{proof}


\begin{lemm}\label{lem: general estimation}
Let $K/k$ be a field extension and let $\nu$ be any discrete valuation of $K$ over $k$.
Let $W_{\sbullet}$ be a graded linear series of $K/k$ of finite type and let $(f_iT^{d_i})_{i=1}^r$ be a system of generators of $W_{\sbullet}$ over $k$.
Set
\begin{equation}
a:=\min\left\{\frac{\nu(f_1)}{d_1},\dots,\frac{\nu(f_r)}{d_r}\right\}.
\end{equation}
Then $W_n\subset\{\phi\in K\,:\,\nu(\phi)\geqslant na\}$ for every $n\in\mathbb N$.
\end{lemm}

\begin{proof}
Any element in $W_n$ can be written in the form
\[
\sum_{d_1n_1+\dots+d_rn_r=n}\alpha_{(n_1,\dots,n_r)}f_1^{n_1}\cdots f_r^{n_r}
\]
($\alpha_{(n_1,\dots,n_r)}\in k$). Then
\[
\nu\left(\sum_{d_1n_1+\dots+d_rn_r=n}\alpha_{(n_1,\dots,n_r)}f_1^{n_1}\cdots f_r^{n_r}\right)\geqslant\min\left\{\sum_{i=1}^r n_i\nu(f_i)\right\}\geqslant an.
\]
\end{proof}

\subsubsection{}

Let $R_{\sbullet}$ be a graded ring which is generated as $R_0$-algebra by a finite family of elements in $R_1$ and let $P:=\mathrm{Proj}(R_{\sbullet})$.
For each homogeneous element $a\in R_{\geqslant 1}$, let
\begin{equation}
(R_{\sbullet})_{(a)}:=\left\{\frac{f}{a^p}\,:\,p\in\mathbb N,\;\deg f=p\deg a\right\}
\end{equation}
be the degree $0$ component of the localisation $R_{\sbullet}[1/a]$, and let $D_{\mathrm{Proj}(R_{\sbullet})+}(a):=\Spec((R_{\sbullet})_{(a)})$ denote the affine open subscheme of $\mathrm{Proj}(R_{\sbullet})$ defined by the non-vanishing of $a$.

Set $\mathcal O_P(n):=\widetilde{R(n)_{\sbullet}}$ {(see Notation and conventions {\bf\ref{graded coherent sheaf}})}.
Given an $s\in R_n$, the local sections $s/1\in H^0(D_{P+}(a),\mathcal O_P(n))=(R(n)_{\sbullet})_{(a)}$ for $a\in R_1$ glue up to a global section $\alpha_n(s)\in H^0(P,\mathcal O_P(n))$.
The following lemmas are well-known.

\begin{lemm}[{\cite[Proposition~II.2.7.3]{EGAII}}]\label{lem:prelim1}
Let $M_{\sbullet}$ be a finitely generated graded $R_{\sbullet}$-module.
If $\widetilde{M_{\sbullet}}=0$, then $M_n=\{0\}$ for any sufficiently positive integer $n$.
\end{lemm}

\begin{lemm}\label{lem:prelim2} Let $R_\sbullet$ be a graded ring and $P=\mathrm{Proj}(R_\sbullet)$.
If $R_{\sbullet}$ is essentially integral and is generated as an $R_0$-algebra by finitely many homogeneous elements in $R_1$, then the canonical homomorphism $\alpha_{\sbullet}:R_{\sbullet}\to R(\mathcal O_P(1))_{\sbullet}:=\bigoplus_{n\in\mathbb N}H^0(P,\mathcal O_P(n))$ is injective and any element of $R(\mathcal O_P(1))_{\sbullet}$ is integral over $R_\sbullet$.
\end{lemm}

\begin{proof}
Suppose that $R_{\sbullet}$ is generated as an $R_0$-algebra by $\{a_1,\dots,a_r\}\subset R_1\setminus\{0\}$, where $a_1,\dots,a_r$ are all non zerodivisors in $R_{\geqslant 1}$ since $R_{\sbullet}$ is essentially integral (see Notation and conventions {\bf\ref{essentially integral}}).
Given any $\mathfrak{p}\in P$, one can find an $a_i$ such that $a_i\notin\mathfrak{p}$; hence $(D_{P+}(a_i))_{i\in\{1,\ldots,r\}}$ covers $P$.
Thus, a section in $R(\mathcal O_P(1))_{\sbullet}$ can naturally be identified with an element in
\begin{equation}\label{eqn:prelim2:1}
\bigcap_{i=1}^rR_{\sbullet}[1/a_i],
\end{equation}
where the intersection is taken in $R_{\sbullet}[1/(a_1\dots a_r)]$.
In particular, $\alpha_{\sbullet}$ is injective.

Given any homogeneous element $u\in R(\mathcal O_P(1))_{\sbullet}$, one can find an $e\geqslant 1$ such that $a_i^eu\in R_{\sbullet}$ for every $i$ by \eqref{eqn:prelim2:1}.
Since $a_1,\dots,a_r$ generates $R_{\geqslant 1}$, one obtains $R_{\geqslant re}u\subset R_{\geqslant re}$.
Moreover, by induction,
\[
R_{\geqslant re}u^n\subset R_{\geqslant re}u^{n-1}\subset\dots\subset R_{\geqslant re}u\subset R_{\geqslant re}
\]
for every $n\geqslant 1$.
It implies that $R_{\sbullet}[u]\subset (1/a_1)^{re}R_{\sbullet}$; hence $u$ is integral over $R_{\sbullet}$ (see for example \cite[Theorem~9.1]{Matsumura}).
\end{proof}

\begin{lemm}\label{lem:prelim3} We keep the notation of Lemma \ref{lem:prelim2}.
Suppose that $R_{\sbullet}$ is a Noetherian integral domain and is generated as an $R_0$-algebra by finitely many homogeneous elements in $R_1$.
\begin{enumerate}[(1)]
\item If $R_{\sbullet}$ is an $N$-1 ring, then there exists an $n_0\geqslant 0$ such that $\alpha_n$ is isomorphic for every $n\geqslant n_0$.
\item If $R_{\sbullet}$ is an integrally closed domain, then $\alpha_n$ is isomorphic for every $n\geqslant 0$.
\end{enumerate}
\end{lemm}

\begin{proof}
(1) Recall that an integral domain is called an $N$-1 ring if its integral closure in its fraction field is a finite generated module over itself. Note that the graded rings $R_\sbullet$ and $R_\sbullet':=R(\mathcal O_P(1))_{\sbullet}$ have the same homogeneous fraction field, which is the field of rational functions of the scheme $\mathrm{Proj}(R_\sbullet)$. In particular, any homogeneous element of $R_\sbullet'$ belongs to the homogeneous fraction field of $R_\sbullet$, which is contained in the fraction field of $R_\sbullet$. By Lemma \ref{lem:prelim2} we obtain that $R_\sbullet'$ is contained in the integral closure of $R_\sbullet$ and hence is a module of finite type over $R_\sbullet$ by the Noetherian and $N$-1 hypotheses.

We consider the exact sequence of $\mathcal O_{\mathrm{Proj}(R_{\sbullet})}$-modules:
\[\xymatrix{
0\ar[r]&\widetilde{\mathrm{Ker}(\alpha_{\sbullet})}\ar[r]&\widetilde{R_{\sbullet}}\ar[r]^-{\widetilde{\alpha_{\sbullet}}}&\widetilde{R_{\sbullet}'}\ar[r]&\widetilde{\mathrm{Coker}(\alpha_{\sbullet})}\ar[r]& 0.}
\]
Since $\widetilde{\alpha_{\sbullet}}$ is isomorphic by \cite[Proposition~II.2.7.11]{EGAII}, we have $\widetilde{\mathrm{Ker}(\alpha_{\sbullet})}=\widetilde{\mathrm{Coker}(\alpha_{\sbullet})}=0$.
Hence, by Lemma \ref{lem:prelim1}, we conclude.

(2) If $R_\sbullet$ is integrally closed, the above argument actually leads to $R_\sbullet=R_\sbullet'$ since $R_\sbullet'$ is contained in the integral closure of $R_\sbullet$.
\end{proof}

\subsection{Proof of Theorem \ref{Thm:Zariski projective}}

Let $X$ and $X'$ be integral normal $k$-schemes with a fixed inclusion $\mathrm{Rat}(X')\subset\mathrm{Rat}(X)$.
Each point $\xi\in X^{(1)}\cup X^{(0)}$ (respectively, $\xi'\in {X'}^{(1)}\cup{X'}^{(0)}$) defines the discrete valuation $\ord_{\xi}$ (respectively, $\ord_{\xi'}$) of $\mathrm{Rat}(X)$ (respectively, of $\mathrm{Rat}(X')$) over $k$.
We define two sets of points on $X$ and on $X'$, respectively, as
\begin{equation}
\mathfrak A_{X/X'}:=\left\{\xi\in X^{(1)}\,:\,\begin{array}{l}\text{$\ord_{\xi}|_{\mathrm{Rat}(X')}$ is not equivalent to any} \\ \text{of $\ord_{\xi'}$ for $\xi'\in {X'}^{(1)}\cup {X'}^{(0)}$}\end{array}\right\}
\end{equation}
and
\begin{equation}
\mathfrak B_{X/X'}:=\left\{\xi'\in {X'}^{(1)}\,:\,\begin{array}{l}\text{$\ord_{\xi'}$ is not equivalent to any} \\ \text{of $\ord_{\xi}|_{\mathrm{Rat}(X')}$ for $\xi\in X^{(1)}$}\end{array}\right\}.
\end{equation}

\begin{lemm}\label{lem: finiteness of the error points}
Let $X$ and $X'$ be integral normal $k$-schemes of finite type with a fixed inclusion $\mathrm{Rat}(X')\subset\mathrm{Rat}(X)$.
\begin{enumerate}
\item The sets $\mathfrak A_{X/X'}$ and $\mathfrak B_{X/X'}$ are both finite.
\item If the inclusion $\mathrm{Rat}(X')\subset\mathrm{Rat}(X)$ is induced from a surjective and flat morphism $\pi:X\to X'$, then both $\mathfrak A_{X/X'}$ and $\mathfrak B_{X/X'}$ are empty.
\item If $X'$ is proper over $k$ and the inclusion $\mathrm{Rat}(X')\subset\mathrm{Rat}(X)$ is induced from a proper birational morphism $\pi:X\to X'$, then $\mathfrak B_{X/X'}=\emptyset$ and $\mathfrak A_{X/X'}$ is the set of the exceptional divisors of $\pi$.
\end{enumerate}
\end{lemm}

\begin{proof}
2: Let $\xi\in X^{(1)}$. Then by \cite[Proposition IV.6.1.1]{EGAIV_2} we have
\[
\dim\mathcal O_{X',\pi(\xi)}=\dim\mathcal O_{X,\xi}-\dim\mathcal O_{\pi^{-1}(\pi(\xi)),\xi}=\text{0 or 1}.
\]
Hence $\pi(\xi)\in {X'}^{(1)}\cup {X'}^{(0)}$ and $\ord_{\xi}|_{\mathrm{Rat}(X')}$ is equivalent to $\ord_{\pi(\xi)}$ by Lemma \ref{lem: centre of the restriction}.

Let $\xi'\in {X'}^{(1)}$.
Given any irreducible component $Z$ of $\pi^{-1}(\overline{\{\xi'\}})$, the generic point $\xi$ of $Z$ is mapped to $\xi'$ via $\pi$ (see \cite[Proposition IV.2.3.4]{EGAIV_2}).
Hence $\ord_{\xi'}$ is equivalent to $\ord_{\xi}|_{K'}$.

1: The inclusion $\mathrm{Rat}(X')\subset \mathrm{Rat}(X)$ yields a $k$-morphism $\pi:U\to X'$, where $U$ denotes a nonempty open subscheme of $X$.
By the theorem of generic flatness \cite[Th\'eor\`eme IV.6.9.1]{EGAIV_2}, there exists a nonempty open subscheme $U'\subset X'$ such that
\[
\bar{\pi}:=\pi|_{\pi^{-1}(U')}:\bar{U}:=\pi^{-1}(U')\to U'
\]
is flat.
Moreover, since $\bar{\pi}$ is an open morphism (see \cite[Th\'eor\`eme IV.2.4.6]{EGAIV_2}), we may assume that $\bar{\pi}$ is surjective.
By the assertion 1 above, $\mathfrak A_{X/X'}$ (respectively, $\mathfrak B_{X/X'}$) is contained in the set consisting of the generic points of the irreducible components of $X\setminus\pi^{-1}(U')$ (respectively, $X'\setminus U'$).

3: By the valuative criterion of properness, there exists an open subscheme $U'\subset X'$ such that $\mathrm{codim}(X'\setminus U',X')\geqslant 2$ and the identification $\mathrm{Rat}(X')=\mathrm{Rat}(X)$ induces an open immersion $U'\to X$.
Hence $\mathfrak B_{X/X'}=\emptyset$ and $\mathfrak A_{X/X'}$ is contained in the exceptional locus of $\pi$.
If $\xi$ is a generic point of an irreducible component of the exceptional locus of $\pi$, then $\pi(\xi)=c_{X'}(\ord_{\xi}|_{\mathrm{Rat}(X')})$ by Lemma \ref{lem: centre of the restriction} and $\dim\mathcal O_{X',\pi(\xi)}$ is $\geqslant 2$.
Hence $\xi\in\mathfrak A_{X/X'}$.
\end{proof}

We restate Theorem \ref{Thm:Zariski projective} as follows.

\begin{theo}
Let $K/K'/k$ be field extensions of finite type and $W_{\sbullet}$ a graded linear series of $K/k$ that is generated over $k$ by the homogeneous elements of degree $1$.
We assume that $W_1$ contains $1\in K$ and that the projective spectrum $P:=\mathrm{Proj}(W_{\sbullet})$ is a normal scheme.
Let $X$ be any integral normal projective $k$-scheme whose field of rational functions is $k$-isomorphic to $k(W_{\sbullet}\cap K'[T])$.
\begin{enumerate}
\item There then exists a $\mathbb Q$-Weil divisor $D$ on $X$ such that
\[
W_n\cap K'\subset H^0(X,nD)\subset k(W_{\sbullet}\cap K'[T])
\]
for every sufficiently positive $n$.
\item If $\mathfrak A_{P/X}=\emptyset$, then there exists a $\mathbb Q$-Weil divisor $D$ on $X$ such that
\[
W_n\cap K'=H^0(X,nD)\subset k(W_{\sbullet}\cap K'[T])
\]
for every sufficiently positive $n$.
\end{enumerate}
\end{theo}

\begin{proof}
Without loss of generality, we may assume that $K=k(W_\sbullet)$ and $K'=\mathrm{Rat}(X)$. In particular, $K$ naturally identifies with the field of rational functions on $P$.
First, we give a valuation-theoretic interpretation of the required statement.
Let $H$ be the effective Cartier divisor on $P$ defined by the image of $1$ via $W_1\to H^0(P,\mathcal O_P(1))$.
By Lemma \ref{lem:prelim3}(1), one has
\begin{align}
W_n &=\left\{\phi\in K\,:\,nH+(\phi)\geqslant 0\right\} \\
&=\left\{\phi\in K\,:\,\text{$\ord_{\xi}(\phi)\geqslant -n\mult_{\xi}(H)$, $\forall$ $\xi\in P^{(1)}$}\right\} \nonumber
\end{align}
for every $n\gg 0$.
Therefore,
\begin{equation}\label{eqn:valuation:theoretic:interpretation}
W_n\cap K'=\left\{\phi\in K'\,:\,\text{$\ord_{\xi}|_{K'}(\phi)\geqslant -n\mult_{\xi}(H)$, $\forall$ $\xi\in P^{(1)}$}\right\}
\end{equation}
for $n\gg 0$.

Next, for each $\xi'\in X^{(1)}$, we define a nonnegative rational number $a_{\xi'}$ as follows.
If $\xi'\notin\mathfrak B_{P/X}$, then we fix an arbitrary point $\xi\in P^{(1)}$ such that $\ord_{\xi}|_{\mathrm{Rat}(X)}$ is equivalent to $\ord_{\xi'}$.
Let $e_\xi$ denote the ramification index of $\ord_{\xi}$ with respect to $K/K'$ (see Notation and conventions \ref{NC:valuations}).
We then set
\[
a_{\xi'}:=e_{\xi}^{-1}\mult_{\xi}(H).
\]
Otherwise, we fix an arbitrary discrete valuation $\nu_{\xi'}$ of $K$ extending $\ord_{\xi'}$, whose existence is assured by Lemma \ref{lem: extension of valuation}, and set
\[
a_{\xi'}:=-\min\{0,\nu_{\xi'}(f_1),\dots,\nu_{\xi'}(f_r)\},
\]
where $\{f_1T,\dots,f_rT\}$ denotes a system of generators of $W_{\sbullet}$ as a $k$-algebra. We define $D:=\sum_{\xi'\in X^{(1)}}a_{\xi'}\overline{\{\xi'\}}$. By the finiteness of $\mathfrak B_{P/X}$ proved in Lemma \ref{lem: finiteness of the error points}, $D$ is well-defined as a $\mathbb Q$-Weil divisor on $X$. Moreover, $D$ is effective and we have $W_n\cap K'\subset H^0(X,nD)$ for every $n\gg 0$ by \eqref{eqn:valuation:theoretic:interpretation} and Lemma \ref{lem: general estimation}.

Lastly, we consider the case where $\mathfrak A_{P/X}=\emptyset$.
Given a $\xi'\in X^{(1)}$, we define a nonnegative rational number $b_{\xi'}$ as follows.
If $\xi'\notin\mathfrak B_{P/X}$, then we set
\[
b_{\xi'}:=\min\left\{e_{\xi}^{-1}\mult_{\xi}(H)\,:\,\begin{array}{l}\text{$\xi\in P^{(1)}$, $e_{\xi}\neq 0$, and $\ord_{\xi}|_{\mathrm{Rat}(X)}$ is} \\ \text{equivalent to $\ord_{\xi'}$}\end{array}\right\}.
\]
Otherwise, we fix a discrete valuation $\nu_{\xi'}$ extending $\ord_{\xi'}$, and set
\[
b_{\xi'}:=-\min\{0,\nu_{\xi'}(f_1),\dots,\nu_{\xi'}(f_r)\}
\]
in the same way as above. If we set $D':=\sum_{\xi'\in X^{(1)}}b_{\xi'}\overline{\{\xi'\}}$, then, since $\mathfrak A_{P/X}=\emptyset$,
\begin{align*}
W_n\cap K'&=\left\{\phi\in\mathrm{Rat}(X)\,:\,\text{$\ord_{\xi'}(\phi)\geqslant -nb_{\xi'}$, $\forall$ $\xi'\in X^{(1)}\setminus\mathfrak B_{P/X}$}\right\} \\
&\supset H^0(X,nD')
\end{align*}
for every $n\gg 0$.
The reverse inclusion follows from the same argument as above.

\end{proof}

In the following, we give an alternative proof for Theorem \ref{Thm:subfinite graded} by using the projective version of Zariski's result (Theorem \ref{Thm:Zariski projective}).

\begin{coro}\label{thm:chens:question}
Let $K/k$ be a finitely generated field extension and $K'/k$ a subextension of $K/k$.
Let $V_{\sbullet}$ be a graded linear series of $K'/k$.
If $V_{\sbullet}$ is contained in a graded linear series $W_{\sbullet}$ of $K/k$ and of finite type over $k$, then $V_{\sbullet}$ is contained a graded linear series $W_{\sbullet}'$ of $K'/k$ and of finite type over $k$.
\end{coro}

\begin{proof}
We divide the proof into three steps.

{\it Step 1:}
In this step, we make several reductions of the theorem.
By the same arguments as in the step 1 of Theorem \ref{Thm: graded case}, we can assume that $V_1$ contains $1$.

\begin{enonce}{Claim}\label{lem:chen}
By enlarging $K$ if necessary, we can assume that $W_{\sbullet}$ is generated by $W_1$ over $k$.
\end{enonce}

\begin{proof}[Proof of Claim \ref{lem:chen}]
Let $f_1T^{d_1},\dots,f_rT^{d_r}\in W_{\geqslant 1}$ be homogeneous generators of $W_{\sbullet}$ over $k$.
Let $T_1,\dots,T_r$ be variables with $\deg T_i=1$ for every $i$.
One can find a homogeneous prime ideal $\mathfrak{p}$ of $W_{\sbullet}[T_1,\dots,T_r]$ such that $\mathfrak{p}$ contains 
\[
I:=(T_1^{d_1}-f_1T^{d_1},\dots,T_r^{d_r}-f_rT^{d_r})
\]
and such that $\mathfrak{p}\cap V_{\sbullet}=\{0\}$.
In fact, let $W_{\sbullet}':=W_{\sbullet}[T_1,\dots,T_r]/I$ and let $a$ be a homogeneous element of degree $\geqslant 1$.
Since the morphism $\Spec\left((W_{\sbullet}')_{(a)}\right)\to\Spec\left((V_{\sbullet})_{(a)}\right)$
is dominant (Lemma \ref{Lem:injective yields dominant}), there exists a homogeneous prime ideal $\mathfrak{p}\in\mathrm{Proj}(W_{\sbullet}')$ such that $\mathfrak{p}\cap V_{\sbullet}=\{0\}$.
We set $U_{\sbullet}:=W_{\sbullet}'/\mathfrak{p}$.
Then $U_{\sbullet}$ is a graded linear series of $k(U_{\sbullet})/k$, $W_{\sbullet}\to U_{\sbullet}$ is injective, and $U_{\sbullet}$ is generated by $U_1=W_1+W_0T_1+\dots+W_0T_r$.
\end{proof}

In particular, we can assume that $P:=\mathrm{Proj}(W_{\sbullet})$ is a projective scheme over $k$ and that $\mathcal O_P(1):=\widetilde{W(1)_{\sbullet}}$ is an invertible sheaf on $P$.

{\it Step 2:}
Let $u:\widehat P\to P$ be a normalisation and $H$ the Cartier divisor defined by the image of $1$ via $V_1\to H^0(\widehat P,u^*\mathcal O_P(1))$.
We choose a very ample divisor $\widehat H$ such that $\widehat H-H$ is effective and such that $R(\widehat H)_{\sbullet}$ is generated by $R(\widehat H)_1T$ over $R(\widehat H)_0$.

Note that the graded $k$-algebra
\[
\widehat W_{\sbullet}:=k\oplus\bigoplus_{n\geqslant 1} H^0(\widehat P,n\widehat H)T^n
\]
is a graded linear series of $K/k$ and of finite type over $k$ (Lemma \ref{Lem: de type fini corps}) and that $\mathrm{Proj}(\widehat W_{\sbullet})$ is isomorphic to $\widehat P$ over $k$.

Applying Theorem \ref{Thm:Zariski projective} to $\widehat W_{\sbullet}$ and $K'/k$, we can find an integral normal projective $k$-scheme $X$, an effective $\mathbb Q$-divisor $D$ on $X$, and an integer $n_0\geqslant 1$ such that $\mathrm{Rat}(X)\subset K'$ and such that $V_n\subset R(\widehat H)_n\cap K'\subset H^0(X,nD)$ for every $n$ with $n\geqslant n_0$.

{\it Step 3:}
Let $\widehat D$ be a very ample divisor on $X$ such that $\widehat D-D$ is effective and such that $R(\widehat D)_{\sbullet}$ is finitely generated over $k$.
Let $W_{\sbullet}'$ be the graded linear series generated by a basis of
\[
\bigoplus_{n<n_0}V_nT^n
\]
over $k$ and by finite number of generators of $R(\widehat D)_{\sbullet}$ over $k$.
Then $W_n'$ contains $V_n$ for every $n\geqslant 0$ and $W_{\sbullet}'$ is finitely generated over $k$.
\end{proof}

As a consequence of Theorem \ref{Thm:Zariski projective}, we can give an estimate of the following type for graded linear series of subfinite type (see also \cite[Corollary 2.1.38]{LazarsfeldI} and Theorem \ref{Theo:Application Fujita approximation} \emph{infra}).

\begin{coro}
Let $K/k$ be a finitely generated field extension and $V_{\sbullet}$ a graded linear series of $K/k$ and of subfinite type.
Let $d$ be the Kodaira-Iitaka dimension of $V_{\sbullet}$.
If $d$ is nonnegative, then there exist an integral normal projective $k$-scheme $X$ and $\mathbb Q$-Cartier divisors $D,D'$ on $X$ such that the rational function field of $X$ is $k$-isomorphic to $k(V_{\sbullet})$, that both $D$ and $D'$ have Kodaira-Iitaka dimension $d$, and that
\[
H^0(X,nD')\subset V_n\subset H^0(X,nD)\subset k(V_{\sbullet})
\]
for every sufficiently positive $n$ with $V_n\neq\{0\}$.
\end{coro}

\begin{proof}
The existence of $D$ results from the same arguments as in Corollary \ref{thm:chens:question}.
Thus, it suffices to show the existence of $D'$ having the prescribed properties.
By changing the grading of $V_\sbullet$, we may assume that $\{n\in\mathbb N\,:\,V_n\neq\{0\}\}$ generates $\mathbb Z$ as a $\mathbb Z$-module.
Choose any sufficiently positive integer $p_0$ such that $k(V_{p_0})=k(V_{\sbullet})$ (see Lemma \ref{Rem:generation of homogeneous fractions}).
Let $W_{\sbullet}$ be the sub-$k$-algebra of $V_{\sbullet}$ generated by $V_{p_0}$, and set
\[
W_{\sbullet}':=\bigoplus_{n\in\mathbb N}W_{p_0n}.
\]
Let $P:=\mathrm{Proj}(W_{\sbullet}')$ and $\mathcal O_P(1):=\widetilde{W_{\sbullet}'(1)}$.
By Lemma \ref{lem:prelim3}, $W_n'=H^0(P,\mathcal O_P(n))\subset V_{p_0n}$ for every $n\gg 1$.
Let $\nu:\widehat P\to P$ be a normalisation.
Let $p$ be any sufficiently positive integer divisible by $p_0$.
Then one can find an ample divisor $A$ on $\widehat P$ such that
\[
H^0(\widehat P,nA)=H^0(P,\nu_*(\mathcal O_{\widehat P}(nA)))\subset H^0(P,\mathcal O_P(pn/p_0))\subset V_{pn}
\]
for every positive integer $n$ (see \cite[D\'emonstration de Proposition 3.6]{Chen15}).

Repeating the same arguments, one can choose an integral normal projective $k$-scheme $X$, two big Cartier divisors $A,A'$ on $X$, and two coprime positive integers $p,p'$ such that
\[
H^0(X,nA)\subset V_{pn}\text{ and }H^0(X,nA')\subset V_{p'n}
\]
for any positive integer $n$.
Moreover, one can choose an ample $\mathbb Q$-Cartier divisor $D'$ on $X$ and two coprime positive integers $q,q'$ such that $qq'D'$ is integral, that $q$ (resp. $q'$) is divisible by $p$ (resp. $p'$), and that
\[
H^0(X,qnD')\subset H^0(X,(qn/p)A)\subset V_{qn}
\]
and
\[
H^0(X,q'nD')\subset H^0(X,(q'n/p')A)\subset V_{q'n}
\]
hold for every integer $n\in\mathbb N_{\geqslant 1}$.

Since
\[
H^0(X,qnD')\otimes_kH^0(X,q'n'D')\to H^0(X,(qn+q'n')D')
\]
is surjective for any sufficiently positive integers $n,n'$ (see for example \cite[Example 1.2.22]{LazarsfeldI}, which is valid over fields of arbitrary characteristics), we have $H^0(X,nD')\subset V_n$ for every sufficiently positive $n$ (recall the arguments in Lemma \ref{Rem:generation of homogeneous fractions}).
\end{proof}

\begin{coro}[{Fujita \cite[Appendix]{FujitaL}}]\label{Cor:Fujita theorem}
Let $X$ be an integral normal projective $k$-scheme and  $D$ an effective Cartier divisor on $X$.
If the Kodaira-Iitaka dimension of $D$ is 1, then the section ring $R(D)_{\sbullet}$ is finitely generated.
\end{coro}

\begin{proof}
Let $K:=\mathrm{Rat}(X)$ and let $C$ be the smooth projective $k$-curve with rational function field $k$-isomorphic to $K':=k(R(D)_{\sbullet})$.
The inclusion $K'\subset K$ defines a rational map $X\dashrightarrow C$ and, by taking a suitable blow-up $\mu:\widehat X\to X$, one obtains a flat morphism $\pi:\widehat X\to C$.
If we set
\[
E:=\sum_{\xi'\in C^{(1)}}\min\left\{e_{\xi}^{-1}\mult_{\xi}(\mu^*D)\,:\,\xi\in\widehat X^{(1)},\, \xi\mapsto\xi',\, e_{\xi}\neq 0\right\}\xi',
\]
then $H^0(C,nE)=H^0(\widehat X,n\mu^*D)=H^0(X,nD)$ for every $n\gg 0$.
Hence the result is reduced to the classic case of curves.
\end{proof}

\begin{rema}
\begin{enumerate}
\item If $X$ is a surface, Zariski \cite{Zariski62} completely classified the cases where $R(D)_{\sbullet}$ is finitely generated \cite[Theorem~10.6 and Proposition~11.5]{Zariski62}. Later, Fujita \cite{FujitaL} generalised the case where the Kodaira-Iitaka dimension is one to the form of Theorem \ref{Cor:Fujita theorem} by using the Iitaka fibrations.
\item For a nef and big Cartier divisor $D$ on $X$, $R(D)_{\sbullet}$ is finitely generated if and only if $D$ is semiample (see \cite[Theorem~2.3.15]{LazarsfeldI}).
\end{enumerate}
\end{rema}

\subsection{Nagata's counterexamples}

Let $N$ and $r$ be positive integers such that $N\geqslant r$ and let $X_1,\dots,X_N$ and $Y_1,\dots,Y_N$ denotes variables.
Firstly, we consider the affine case as in \cite{Nagata59, Mukai01,MukaiBook,Mukai04}.
Set
\[
W_{\sbullet}:=\mathbb C[X_1,Y_1,\dots,X_N,Y_N]=\bigoplus_{n\in\mathbb N}\mathbb C[X_1,Y_1,\dots,X_N,Y_N]_n,
\]
where $\mathbb C[X_1,Y_1,\dots,X_N,Y_N]_n$ denotes the $\mathbb C$-vector space of the homogeneous polynomials of degree $n$, and let
\[
K:=\mathrm{Frac}(W_{\sbullet})=\mathbb C(X_1,Y_1,\dots,X_N,Y_N).
\]
We take linearly independent linear forms
\begin{align*}
L_1(T_1,\dots,T_N) &=a_{1,1}T_1+\dots+a_{1,N}T_N,\\
&\vdots \\
L_r(T_1,\dots,T_N) &=a_{r,1}T_1+\dots+a_{r,N}T_N.
\end{align*}
For simplicity, we set $a_{1,1}=\dots=a_{1,N}=1$ and assume
\begin{equation}\label{eqn: nonzero minor}
\det(a_{i,j})_{\substack{1\leqslant i\leqslant r \\ N-r+1 \leqslant j\leqslant N}}\neq 0.
\end{equation}
We set
\begin{align*}
Z_0 &:=X_1\cdots X_N,\\
Z_1 &:=Z_0L_1\left(\frac{Y_1}{X_1},\dots,\frac{Y_N}{X_N}\right), \\
&\vdots \\
Z_r &:=Z_0L_r\left(\frac{Y_1}{X_1},\dots,\frac{Y_N}{X_N}\right),
\end{align*}
and set
\[
K':=\mathbb C(Z_0,Z_1,\dots,Z_r).
\]
Note that one can find linear forms $M_1,\dots,M_r$ such that
\begin{align}
X_N &=\frac{Z_0}{X_1\cdots X_{N-1}}, \label{eqn: the inverse relation}\\
Y_{N-r+1} &=X_{N-r+1}M_1\left(\frac{Y_1}{X_1},\dots,\frac{Y_{N-r}}{X_{N-r}},\frac{Z_1}{Z_0},\dots,\frac{Z_r}{Z_0}\right), \nonumber\\
&\vdots\nonumber\\
Y_N &=\frac{Z_0}{X_1\cdots X_{N-1}}M_r\left(\frac{Y_1}{X_1},\dots,\frac{Y_{N-r}}{X_{N-r}},\frac{Z_1}{Z_0},\dots,\frac{Z_r}{Z_0}\right),\nonumber
\end{align}
and that $K=K'(X_1,\dots,X_{N-1},Y_1,\dots,Y_{N-r})$.
Nagata \cite{Nagata59} has proved that, if $N=13$, $r=3$, and $L_1,\dots,L_r$ are sufficiently general, then $W_{\sbullet}\cap K'$ is not of finite type over $\mathbb C$.
Later, Mukai obtained similar results for $N=9$ and $r=3$ (see \cite[section 2.5]{MukaiBook}).
Each $W_n\cap K'$ is nonzero if and only if $N$ divides $n$, and each $F\in W_{Nn}\cap K'$ can be written in the form
\[
F=Z_0^mf(Z_1,Z_2,Z_3),
\]
where $m$ is an integer and $f\in H^0(\mathbb P^2,\mathcal O_{\mathbb P^2}(n-m))$.
In view of the following lemma, we know that the fraction field of $W_{\sbullet}\cap K'$ coincides with $K'$ and that $W_{\sbullet}\cap K'$ is contained in $\mathbb C[Z_0,Z_1/Z_0,Z_2/Z_0,Z_3/Z_0]$.

\begin{lemm}[\text{\cite[Lemma 2.45]{MukaiBook}}]
Let $d\geqslant 1$ be any integer.
For $f\in H^0(\mathbb P^2,\mathcal O_{\mathbb P^2}(d))$ and for each $i\in\{1,2,\dots,N\}$, we have
\[
\ord_{\{X_i=0\}}|_{K'}(Z_0^mf(Z_1,Z_2,Z_3))=m+\ord_{\bm{a}_i}(f(Z_1,Z_2,Z_3)),
\]
where $\bm{a}_i:=(a_{1,i}:\dots:a_{r,i})$ for $i\in\{1,2,\dots,N\}$.
\end{lemm}

Next, we are going to consider a projective variant of Nagata's counterexample.
Let $T$ denote another invariant.
We define a graded linear series of $K/\mathbb C$ as
\[
\widehat W_{\sbullet}:=\bigoplus_{n\in\mathbb N}\mathbb C[X_1,Y_1,\dots,X_N,Y_N]_{\leqslant n}T^n,
\]
where $\mathbb C[X_1,Y_1,\dots,X_N,Y_N]_{\leqslant n}$ denotes the $\mathbb C$-vector space of the polynomials of degree $\leqslant n$.
Note that, for each $n$, $F(X_1,Y_1,\dots,X_N,Y_N)\in\widehat W_n\cap K'$ if and only if $F(X_1,Y_1,\dots,X_N,Y_N)\in W_{\leqslant n}\cap K'$.
Let $P:=\mathrm{Proj}(\widehat W_{\sbullet})$ and
\begin{equation}
X:=\mathrm{Proj}\left(\mathbb C[T^N,Z_0T^N,Z_1T^N,Z_2T^N,Z_3T^N]\right).
\end{equation}
Let $H:=\{T^N=0\}$ and $D_0:=\{Z_0=0\}$. Then we have
\[
\widehat W_{Nn}\cap K' \subset H^0(X,nH+nND_0)
\]
and $H^0(X,nH+nND_0)$ is the $\mathbb C$-vector space generated by $\{Z_0^{-m}f(Z_1,Z_2,Z_3)\,:\,f\in H^0(\mathbb P^2,\mathcal O_{\mathbb P^2}(m+n)),\,-n\leqslant m\leqslant nN\}$.

\begin{enonce}{Claim}
We have
\[
\mathfrak A_{P/X}\subset\bigcup_{i=1}^N\{X_i=0\}\quad\text{and}\quad\mathfrak B_{P/X}\subset D_0.
\]
\end{enonce}

\begin{proof}
Note that $\ord_{\{T^N=0\}}=(1/N)\ord_{\{T=0\}}|_{K'}$.
Since
\begin{multline*}
\mathbb C[X_1,\dots,X_N,Y_1,\dots,Y_N][1/Z_0] \\
=\mathbb C[Z_0,Z_1,Z_2,Z_3][1/Z_0][X_1,\dots,X_N,Y_1,\dots,Y_{N-3}]/(X_1\cdots X_N/Z_0-1),
\end{multline*}
$\Spec(\mathbb C[X_1,\dots,X_N,Y_1,\dots,Y_N][1/Z_0])$ is a relative global complete intersection and, hence, is flat over $\Spec(\mathbb C[Z_0,Z_1,Z_2,Z_3][1/Z_0])$. Thus we conclude.
\end{proof}

\begin{exem}
If $r=2$, $N\geqslant 2$, and $a_{2,1},\dots,a_{2,N}$ are mutually distinct, then $\widehat W_{\sbullet}\cap K'[T]$ is finitely generated by $T$, $Z_0T^N$, $Z_1T^N$, $Z_2T^N$, and 
\[
\frac{(a_{2,1}Z_1-Z_2)\cdots (a_{2,N}Z_1-Z_2)}{Z_0}T^{N(N-1)}.
\]
\end{exem}

\begin{rema}
In \cite{Mukai01,Mukai04}, Mukai considered the subfield
\[
K'':=\mathbb C(X_1,\dots,X_N,Z_1,\dots,Z_r)
\]
and studied the finite generation of $\widehat W_{\sbullet}\cap K''[T]$.
In this case, we consider the weighted projective space
\[
\mathrm{Proj}(\mathbb C[T,X_1T,\dots,X_NT,Z_1T^N,\dots,Z_rT^N]).
\]
Let $D_i$ (respectively, $H$) be the hypersurface defined by $X_iT$ for $i\in\{1,2,\dots,N\}$ (respectively, $T$).
One then has
\[
\widehat W_{\sbullet}\cap K''[T]\subset R(D_1+\dots+D_N+H)_{\sbullet}.
\]
\end{rema}

\section{Applications}

In this section, we apply the subfinite criterion (Theorem \ref{Thm:subfinite graded}) to the study of Fujita approximation for general subfinite graded linear series. Throughout the section, we let $k$ be a field and $K/k$ be a finitely generated field extension.

\begin{defi}
Let $V_\sbullet$ be a graded linear series of $K/k$ and $d$ be its Kodaira-Iitaka dimension (see Definition \ref{Def: Kodaira-Iitaka}). If $d\neq -\infty$, we define the \emph{volume} of $V_\sbullet$ as
\begin{equation}
\mathrm{vol}(V_\sbullet):=\limsup_{n\rightarrow+\infty}\frac{\dim_k(V_n)}{n^{d}/d!}.
\end{equation}
A priori this invariant takes value in $[0,+\infty]$. We will see below that, if in addition the graded linear series $V_\sbullet$ is of subfinite type, then its volume is always a positive real number. 

We say that a graded linear series $V_\sbullet$  \emph{satisfies the Fujita approximation property} if 
\[\sup_{\begin{subarray}{c}W_\sbullet\subset V_\sbullet\\
W_\sbullet\text{ of finite type}\\
\dim(W_\sbullet)=\dim(V_\sbullet)
\end{subarray}}\mathrm{vol}(W_\sbullet)=\mathrm{vol}(V_\sbullet),\]
where $W_\sbullet$ runs over the set of all graded linear series of finite type which are contained in $V_\sbullet$ and such that $W_\sbullet$ has the same Kodaira-Iitaka dimension as $V_\sbullet$.
\end{defi}  

The purpose of the section is to establish the following approximation result. 
\begin{theo}\label{Theo:Application Fujita approximation}
Any graded linear series $V_\sbullet$ of $K/k$ which is of subfinite type and has nonnegative Kodaira-Iitaka dimension $d$ satisfies the Fujita approximation property. Moreover, one has
\[\mathrm{vol}(V_\sbullet)=\lim_{n\in\mathbb N(V_\sbullet),\,n\rightarrow+\infty}\frac{\dim_k(V_n)}{n^{d}/d!}\in (0,+\infty),\]
where $\mathbb N(V_\sbullet)=\{n\in\mathbb N\,:\,V_n\neq\{0\}\}$.
\end{theo}
\begin{proof}By changing the grading we may assume without loss of generality that $V_n\neq\{0\}$ for sufficiently positive integer $n$. Let $K'$ be the homogeneous fraction field $k(V_\sbullet)$. Note that $K'/k$ is a subextension of $K/k$ and hence is finitely generated. Moreover, by Theorem \ref{Thm:subfinite graded}, we obtain that $V_\sbullet$ viewed as a graded linear series of $K'/k$ is of subfinite type. Therefore, the assertions follow from \cite[Theorem~1.1]{ChenOk} (by definition $V_\sbullet$ is birational if we consider it as a graded linear series of $K'$). 
\end{proof}

\begin{rema}In the case where the field $K$ admits a valuation \emph{of one-dimensional leaves} in a totally ordered abelian group of finite type (this is the case notably when $k$ is an algebraically closed field), we recover a result of Kaveh and Khovanskii \cite[Corollary 3.11 (2)]{Kaveh_Khovanskii12b}. Note that the existence of a valuation of one-dimensional leaves on $V_\sbullet$ implies that $V_\sbullet$ is geometrically integral since such a valuation induces by extension of scalars a valuation of one-dimensional leaves on $V_\sbullet\otimes_kk'$ for any extension of fields $k'/k$. In particular, for any pair of homogeneous elements $x$ and $y$ of $V_\sbullet\otimes_kk'$, the valuation of $xy$ is equal to the sum of the valuations of $x$ and $y$, which implies that $V_\sbullet\otimes_kk'$ is an integral domain. 
\end{rema}

By combining the results of \cite{Chen15} and the subfiniteness result (Theorem \ref{Thm:subfinite graded}), we obtain the following upper bound for the Hilbert-Samuel function of general graded linear series of subfinite type.

\begin{theo}\label{Thm: effective upper bound}
Let $V_\sbullet$ be a graded linear series of $K/k$ and $d$ its Kodaira-Iitaka dimension. There then exists a function $f:\mathbb N\rightarrow\mathbb R_+$ such that
\[f(n)=\mathrm{vol}(V_\sbullet)\frac{n^d}{d!}+O(n^{d-1}),\quad n\rightarrow+\infty\]
and 
\[\forall\,n\in\mathbb N,\quad \mathrm{dim}_k(V_n)\leqslant f(n).\]
\end{theo}

\begin{rema}\label{Rem:Okounkov body}
The result \cite[Theorem~1.1]{ChenOk} actually provides more geometric information about the graded linear series of subfinite type. Let $K/k$ be a finitely generated transcendental field extension and let $d$ be the transcendence degree of $K/k$. We fix a flag
\[k=K_0\subset K_1\subset\ldots \subset K_d=K\]
of subfields of $K$ containing $k$ such that each extension $K_i/K_{i-1}$ is transcendental and has transcendence degree $1$. Let $\mathcal A(K/k)$ be the set of all graded linear series of subfinite type $V_\sbullet$ of $K/k$ such that $k(V_\sbullet)=k$. Then there has been constructed in \cite{ChenOk} a map  $\Delta$ from $\mathcal A(K/k)$ to the set of convex bodies in $\mathbb R^d$ which satisfies the following conditions.
\begin{enumerate}[(a)]
\item If $V_\sbullet$ and $V_\sbullet'$ are two graded linear series in $\mathcal A(K/k)$ such that $V_\sbullet\subset V_\sbullet'$, then one has $\Delta(V_\sbullet)\subset\Delta(V_\sbullet')$.
\item If $V_\sbullet$ and $W_\sbullet$ are two graded linear series in $\mathcal A(K/k)$, then 
\[\Delta(V_\sbullet\cdot W_\sbullet)\supset\Delta(V_\sbullet)+\Delta(W_\sbullet):=\{x+y\,:\,x\in\Delta(V_\sbullet),\, y\in\Delta(W_\sbullet)\},\]
where $V_\sbullet\cdot W_\sbullet$ denotes the graded linear series whose $n$-th homogeneous component is the $k$-vector space generated by $\{fg\,:\,f\in V_n,\;g\in W_n\}$.
\item For any graded linear series $V_\sbullet$ in $\mathcal A(K/k)$, the volume of $V_\sbullet$ identifies with the Lebesgue measure of $\Delta(V_\sbullet)$ multiplied by $d!$.
\end{enumerate}
This allows us to construct the arithmetic analogue of Newton-Okounkov bodies for general arithmetic graded linear series of subfinite type, using the ideas of \cite{Boucksom_Chen}.
\end{rema}

In what follows, we assume that $k$ is a number field. We denote by $M_k$ the set of all places of $k$. For each $v\in M_k$, let $|\ndot|_v$ be an absolute value on $k$ which extends either the usual absolute value or certain $p$-adic absolute value (so that $|p|_v=p^{-1}$) on $\mathbb Q$. 

As \emph{adelic vector bundle} on $\Spec k$, we refer to the data $\overline V=(V,(\|\ndot\|_v)_{v\in M_k})$ of a finite dimensional vector space $V$ over $k$ and a family of norms $\|\ndot\|_v$ over $V\otimes_kk_v$ such that there exists a basis $(e_i)_{i=1}^r$ of $V$ over $k$ and a finite subset $S$ of $M_k$ satisfying the following condition:
\[\forall\,v\in M_k\setminus S,\quad\forall\,(\lambda_1,\ldots,\lambda_r)\in k_v^r,\quad \|\lambda_1e_1+\cdots+\lambda_re_r\|_v=\max_{i\in\{1,\ldots,r\}}|\lambda_i|_v.\] 
Given an adelic vector bundle $\overline V$ on $\Spec k$, for any nonzero element $s\in V$, we define the \emph{Arakelov degree of $s$} as
\[\widehat{\deg}(s):=-\sum_{v\in M_k}[k_v:\mathbb Q_v]\ln\|s\|_v.\]
By the product formula
\[\forall\,a\in k^{\times},\quad\sum_{v\in M_k}[k_v:\mathbb Q_v]\ln|a|_v=0\]
we obtain that
\[\forall\,a\in k^{\times},\quad \widehat{\deg}(as)=\widehat{\deg}(s).\]
Moreover, the \emph{Arakelov degree of $\overline V$} is defined as
\[-\sum_{v\in M_k}\ln\|\eta\|_{v,\det},\]
where $\eta$ is a nonzero element of $\det(V)$, and 
\[\|\eta\|_{v,\det}=\inf\{\|x_1\|_v\cdots\|x_r\|_v\,:\,\eta=x_1\wedge\cdots\wedge x_r\}.\] Again by the product formula we obtain that the definition does not depend on the choice of $\eta\in\det(V)\setminus\{0\}$.

Let $\overline V$ be an adelic vector bundle of rank $r$ on $\Spec k$. For any $t\in\mathbb R$, let 
\[\mathcal F^t(V)=\mathrm{Vect}_k(\{s\in V\setminus\{0\}\,:\,\widehat{\deg}(s)\geqslant t\}).\]
This is a decreasing $\mathbb R$-filtration on $V$, called the \emph{$\mathbb R$-filtration by minima}.
Note that for any $i\in\{1,\ldots,r\}$, the number
\[\lambda_i(\overline V)=\sup\{t\in\mathbb R\,:\,\mathrm{rk}_k(\mathcal F^t(V))\geqslant i\}\]
coincides with the minus logarithmic version of the $i$-th minima in the sense of Roy and Thunder \cite{Roy_Thunder96,Roy_Thunder99}. For any $s\in V$, we let
\[\lambda(s):=\sup\{t\in\mathbb R\,:\,s\in\mathcal F^t(V)\}.\]

In the following, we let $K/k$ be a finitely generated field extension of the number field $k$. Let $V_\sbullet$ be a graded linear series of subfinite type of $K/k$. For each $n\in\mathbb N$, we equip $V_n$ with a structure of adelic vector bundle $(V_n,(\|\ndot\|_{n,v})_{v\in M_k})$ on $\Spec k$ such that, for any $v\in M_k$,
\begin{equation}\label{Equ:submultiplicative}\forall\,(n,m)\in\mathbb N^2,\;\forall\,(s_n,s_m)\in V_{n}\times V_{m},\quad \|s_n\cdot s_m\|_v\leqslant\|s_n\|_v\cdot\|s_m\|_v.\end{equation}
We assume in addition that 
\[\lambda_{\max}(\overline V_\sbullet):=\limsup_{n\rightarrow+\infty}\frac{\lambda_1(\overline V_n)}{n}<+\infty.\]
This condition implies that $V_\sbullet$ has a nonnegative Kodaira-Iitaka dimension.
For any $t\in\mathbb R$, let \[V^{t}_\sbullet:=\bigoplus_{n\in\mathbb N}\mathcal F^{nt}(V_n).\]
It is a graded linear series of $K/k$. By definition one has $V_n^t=\{0\}$ if $n\in\mathbb N_{\geqslant 1}$ and $t>\lambda_{\max}(\overline V_\sbullet)$.

\begin{prop}\label{Pro: same homogeneous fraction}
For any $t<\lambda_{\max}(\overline V_\sbullet)$, one has $k(V_\sbullet)=k(V_\sbullet^t)$.  
\end{prop}
\begin{proof} Clearly one has $k(V_\sbullet)\supset k(V_\sbullet^t)$. It suffices to prove the converse inclusion.
Let $n\geqslant 1$ be an integer and $f$, $g$ be nonzero elements in $V_n$. Since $t<\lambda_{\max}(V_\sbullet)$ there exist $m\in\mathbb N_{\geqslant 1}$ and $s\in V_m$ such that $\lambda(s)> mt$. Thus for sufficiently positive integer $\ell$ one has $\lambda(s^\ell f)>(\ell m+n)t$ and $\lambda(s^\ell g)>(\ell m+n)t$. Therefore $\{s^\ell f, s^\ell g\}\subset V_{\ell m+n}^t$, which implies $f/g\in k(V_\sbullet^t)$.
\end{proof}

The above proposition allows us to consider $V_\sbullet^t$ as a birational graded linear series of $k(V_\sbullet)/k$ and to construct its Newton-Okounkov body as reminded in Remark \ref{Rem:Okounkov body}. We define the \emph{concave transform} of $\overline{V}_\sbullet$ as the function $G_{\overline{V}_\sbullet}$ on $\Delta(V_\sbullet)$ sending $x\in\Delta(V_\sbullet)$ to
\[\sup\{t<\lambda_{\max}(\overline{V}_\sbullet)\,:\,x\in\Delta(V_\sbullet^t)\}.\] By the condition (b) in Remark \ref{Rem:Okounkov body}, the function $G_{\overline{V}_\sbullet}$ is concave.

The following result generalises \cite[Theorem~2.8]{Boucksom_Chen} to the case of subfinite adelically normed graded linear series.

\begin{theo}\label{Thm:convergence thm}
Let $K/k$ be a finitely generated extension of a number field $k$, and $\overline{V}_\sbullet=\bigoplus_{n\in\mathbb N}\overline V_n$ a graded linear series of subfinite type of $K/k$ of Kodaira-Iitaka dimension $d\geqslant 0$, equipped with structures of adelic vector bundles on $\Spec k$, which satisfy the submultiplicativity condition \eqref{Equ:submultiplicative} and the condition $\lambda_{\max}(\overline{V}_\sbullet)<+\infty$. Then the sequence of measures
\[\frac{1}{\mathrm{rk}_k(V_n)}\sum_{i=1}^{\mathrm{rk}_k(V_n)}\delta_{\lambda_i(V_n)/n},\quad n\in\mathbb N(V_\sbullet)=\{m\in\mathbb N\,:\,V_m\neq\{0\}\}\]
converges weakly to a Boreal probability measure on $\mathbb R$, which is the image of the uniform measure \[\frac{1}{\mathrm{vol}(\Delta(V_\sbullet))}\indic_{\Delta(V_\sbullet)}(x)\,\mathrm{d}x\] by the concave transform $G_{\overline{V}_\sbullet}$.
\end{theo}
\begin{proof}
For any $t<\lambda_{\max}(\overline V_\sbullet)$, the graded linear series $V^t_\sbullet$ has the same homogeneous fraction field as $V_\sbullet$ (see Proposition \ref{Pro: same homogeneous fraction}). Hence we can construct a decreasing family $(\Delta(V_\sbullet^t))_{t<\lambda_{\max}(\overline V_\sbullet)}$ of convex bodies contained in $\Delta(V_\sbullet)$, as described in Remark \ref{Rem:Okounkov body}. Moreover, if $t_1$ and $t_2$ are two real numbers which are $<\lambda_{\max}(\overline{V}_\sbullet)$. Then by the same method as in \cite[\S1.3]{Boucksom_Chen}, we obtain the desired result.
\end{proof}


\backmatter
\bibliography{subfinite}
\bibliographystyle{smfplain}

\end{document}